\documentclass[reqno,10pt,centertags]{amsart} 
\usepackage{amsmath,amsthm,amscd,amssymb,latexsym,esint,upref,stmaryrd,
enumerate,color,verbatim,yfonts,enumitem}
\calclayout
\usepackage{hyperref} 
\newcommand*{\mailto}[1]{\href{mailto:#1}{\nolinkurl{#1}}}

\setenumerate{label=$($\textit{\roman*}$)$}

\newcommand{\R}{{\mathbb R}}
\newcommand{\N}{{\mathbb N}}

\newcommand{\C}{{\mathbb C}}

\newcommand{\bbN}{{\mathbb{N}}}

\newcommand{\bbR}{{\mathbb{R}}}

\newcommand{\bbZ}{{\mathbb{Z}}}

\newcommand{\cA}{{\mathcal A}}
\newcommand{\cB}{{\mathcal B}}
\newcommand{\cC}{{\mathcal C}}
\newcommand{\cD}{{\mathcal D}}

\newcommand{\cH}{{\mathcal H}}

\newcommand{\cQ}{{\mathcal Q}}

\newcommand{\cU}{{\mathcal U}}

\renewcommand{\a}{\alpha}
\renewcommand{\b}{\beta}
\newcommand{\g}{\gamma}

\newcommand{\z}{\zeta}

\DeclareMathOperator{\supp}{supp}

\DeclareMathOperator{\ran}{ran}

\renewcommand{\Re}{\text{\rm Re}}
\renewcommand{\Im}{\text{\rm Im}}
\renewcommand{\ln}{\text{\rm ln}}

\newcommand{\norm}[1]{\lVert#1\rVert}

\newcommand{\no}{\notag}
\newcommand{\lb}{\label}
\newcommand{\f}{\frac}

\newcommand{\ol}{\overline}

\newcommand{\wti}{\widetilde}

\newcommand{\dott}{\,\cdot\,}

\renewcommand{\dot}{\overset{\textbf{\Large.}}}

\newcommand{\bi}{\bibitem}

\newcommand{\al}{\alpha}
\newcommand{\be}{\beta}

\newcommand{\Lr}{{L^2((a,b);rdx)}}

\newcommand{\ACl}{{AC_{loc}((a,b))}}
\newcommand{\Ll}{{L^1_{loc}((a,b);dx)}}

\makeatletter
\def\theequation{\@arabic\c@equation}
\allowdisplaybreaks 
\numberwithin{equation}{section}
\usepackage{tikz}
\newtheorem{theorem}{Theorem}[section]
\newtheorem{proposition}[theorem]{Proposition}
\newtheorem{lemma}[theorem]{Lemma}
\newtheorem{corollary}[theorem]{Corollary}
\newtheorem{definition}[theorem]{Definition}
\newtheorem{hypothesis}[theorem]{Hypothesis}
\newtheorem{example}[theorem]{Example}

\theoremstyle{remark}

\newenvironment{remark}[1][]{\refstepcounter{theorem}\par\medskip\noindent\textit{Remark~$\theexample. #1$} \rmfamily}{{\ }\hfill $\diamond$ \vspace{6pt}}%Define remark environment with diamond at the end

\begin{document}

\title[$K$-invariant operators]{Maximally dissipative and self-adjoint extensions of $K$-invariant operators}

\author[C.\ Fischbacher]{Christoph Fischbacher} 
\address[C. Fischbacher]{Department of Mathematics, 
Baylor University, Sid Richardson Bldg, 1410 S.\,4th Street, Waco, TX 76706, USA}
\email{\url{c\_fischbacher@baylor.edu}}

\author[B.\ Rosenzweig]{Bart Rosenzweig}
\address[B.\ Rosenzweig]{Department of Mathematics, The Ohio State University \\
100 Math Tower, 231 West 18th Avenue, Columbus, OH 43210, USA}
\email{\mailto{rosenzweig.26@osu.edu}}
%\email{rosenzweig.26@osu.edu}

\author[J.\ Stanfill]{Jonathan Stanfill}
\address[J.\ Stanfill]{Division of Geodetic Science, School of Earth Sciences, The Ohio State University \\
275 Mendenhall Laboratory, 125 South Oval Mall, Columbus, OH 43210, USA}
\email{\mailto{stanfill.13@osu.edu}}
%\email{stanfill.13@osu.edu}
%\urladdr{\url{https://u.osu.edu/stanfill-13/}}
%\urladdr{https://u.osu.edu/stanfill-13/}

%%%%%%%%%%%%%%%%%%%%%%%%%%%%%%%%%%%%%%%%% 
\date{\today}
\subjclass{Primary: 34B24, 47A05, 47B25; Secondary: 47A10, 47B44, 47B65.}
\keywords{Self-adjoint operators, Friedrichs and Krein--von Neumann extensions, dissipative operators, Sturm--Liouville operators}

%%%%%%%%%%%%%%%%%%%%%%%%%%%%%%%%%%%%%%%%
\begin{abstract}

We introduce the notion of $K$-invariant operators, $S$, (in a Hilbert space) with respect to a bounded and boundedly invertible operator $K$ defined via $K^*SK=S$. Conditions such that self-adjoint and maximally dissipative extensions of $K$-invariant symmetric operators are also $K$-invariant are investigated. In particular, the Friedrichs and Krein--von Neumann extensions of a nonnegative $K$-invariant symmetric operator are shown to always be $K$-invariant, while the Friedrichs extension of a $K$-invariant sectorial operator is as well.
We apply our results to the case of Sturm--Liouville operators where $K$ is given by $(Kf)(x)=A(x)f(\phi(x))$ under appropriate assumptions. Sufficient conditions on the coefficient functions for $K$-invariance to hold are shown to be related to Schr\"oder's equation and all $K$-invariant self-adjoint extensions are characterized. Explicit examples are discussed including a Bessel-type Schr\"odinger operator satisfying a nontrivial $K$-invariance on the half-line.
\end{abstract}
%%%%%%%%%%%%%%%%%%%%%%%%%%%%%%%%%%%%%%%%

\maketitle

%{\scriptsize{\tableofcontents}}
%\normalsize

%%%%%%%%%%%%%%%%%%%%%%%%
%%%%%%%%%%%%%%%%%%%%%%%% 
\section{Introduction}\lb{s1}
%%%%%%%%%%%%%%%%%%%%%%%%
%%%%%%%%%%%%%%%%%%%%%%%%

Let $S$ be a nonnegative symmetric operator in a Hilbert space $\cH$ and $K$ a bounded and boundedly invertible operator such that $S$ exhibits an invariance with respect to $K$ that is of the form $K^*SK=S$ -- a property which we will refer to as $K$-invariance. The main purpose of this paper is to describe all nonnegative self-adjoint extensions $\hat{S}$ of $S$ which are $K$-invariant, that is, which satisfy $K^*\hat{S}K=\hat{S}$. 
For the special case that $K$ is unitary, $S$ being $K$-invariant is equivalent to saying that $S$ and $K$ commute: $SK=KS$.  In \cite{{ILPP14}}, the authors studied the case when $S$ is a symmetric operator with equal defect indices and $K$ unitary such that $SK=KS$. They found that a self-adjoint extension $S_U$ of $S$, where $U$ is the unique unitary map between the defect spaces describing this extension $S_U$ via von Neumann theory, satisfies $S_UK=KS_U$ if and only $KU=UK$. In addition, they obtained first results on the $K$-invariance of quadratic forms and the associated induced self-adjoint operators. Instead of focusing on the von Neumann theory of self-adjoint extensions, our focus lies on properties that need to be required of the auxiliary operator $B:\cD(B)\subseteq \ker(S^*)\rightarrow\overline{\cD(B)}$ within the framework of Birman--Krein--Vishik--Grubb extension theory in order to ensure that the self-adjoint extension $S_B$ described by this auxiliary operator remains $K$-invariant. 

Moreover, we go beyond the unitary setup and allow $K$ to be any bounded and boundedly invertible operator, which need not be unitary anymore. This was motivated by a work of Makarov and Tsekanovskii \cite{MT07}, where so-called $\mu$-scale invariant operators $S$ were studied. The notion of $\mu$-scale invariance means that $S$ satisfies $\mathcal{U}^*S\mathcal{U}=\mu S$ for some unitary operator $\cU$ and a scalar $\mu>0$. One of their main results is that the Friedrichs and Krein--von Neumann extensions of a nonnegative $\mu$-scale invariant symmetric operator maintain this property (see also \cite[Thm. 5]{Be07}).
Now, if $\mu\neq 1$ and one defines the non-unitary operator $K_\mu:=\mu^{-1/2}\mathcal{U}$, then $S$ being $\mu$-scale invariant is equivalent to $S$ satisfying $K_\mu^*SK_\mu=S$, thus falling under the scheme of $K$-invariance studied currently, allowing us to readily extend the results of \cite{MT07}. We also mention further works in this direction \cite{BBNV15,BBUV18}, as well as \cite{HK09} considering so-called $p(t)$-homogeneous operators, of which $\mu$-scale invariant are a special case. Furthermore, despite the great interest in invariance of operators, we are unaware of any examples of invariance studied for general Sturm--Liouville operators with potential such as considered here.

We now turn to the content of each section.
In Section \ref{s2}, we begin by defining what it means for a densely defined, closable operator $T$ to be $K$-invariant, showing that this immediately implies $T^*$ and $\overline{T}$ are as well. From there, we characterize when a restriction of $T^*$ is $K$-invariant in Lemma \ref{lemma:A3}, before applying these results to study nonnegative self-adjoint and maximally dissipative extensions. In particular, we show that the Friedrichs and Krein--von Neumann extensions of a nonnegative $K$-invariant symmetric operator are always $K$-invariant, while the Friedrichs extension of a $K$-invariant sectorial operator is as well. We then turn to the necessary and sufficient conditions for a self-adjoint/maximally dissipative extension to be $K$-invariant in Theorem \ref{thm:A.10}, before showing how the conditions simplify when $K$ is unitary. We end the section by constructing in Theorem \ref{thm:2.13} a class of nonnegative self-adjoint extensions of a strictly positive $K$-invariant symmetric operator, $S$, which are also $K$-invariant, and then investigate additional properties of such extensions whenever $S$ has finite defect index.

In Section \ref{s3}, we then apply our abstract framework to the setting of general Sturm--Liouville operators whenever $K:\Lr\to\Lr$ is given by $(Kf)(x)=A(x) f(\phi(x))$, under appropriate assumptions on $A$ and $\phi$ (see Lemma \ref{lemmabdinv}). Sufficient conditions on the coefficient functions $p,q,r$ for a Sturm--Liouville operator associated with the differential expression $\tau=(1/r(x))[-(d/dx)p(x)(d/dx)+q(x)]$ to be $K$-invariant are shown in Theorem \ref{thm3.5} to be
\begin{align}
\begin{split}\label{1.1}
r(x)&=Cr\big(\phi^{-1}(x)\big),\quad p(x)=\big[A\big(\phi^{-1}(x)\big)\big]^2\phi'\big(\phi^{-1}(x)\big)p\big(\phi^{-1}(x)\big),\\
q(x)&=\frac{A\big(\phi^{-1}(x)\big)}{\phi'\big(\phi^{-1}(x)\big)}\big\{A\big(\phi^{-1}(x)\big) q \big(\phi^{-1}(x)\big)-\big(A^{[1]}\big)'\big(\phi^{-1}(x)\big)\big\}.
\end{split}
\end{align}
We would like to point out that the equation satisfied by $r$ is Schr\"oder's equation \cite{Schr70}, that is, the equation is the eigenvalue equation for the composition operator sending $f$ to $f\big(\phi^{-1}(\dott)\big)$ with eigenvalue $C^{-1}$.
Furthermore, whenever $A=1$, the resulting equation satisfied by $p$ is the so-called Julia's equation \cite{AB15}. In fact, when $A$ is constant, the equation for $1/p$ can be integrated to arrive at the same Schr\"oder's equation as for $r$ but with eigenvalue $A^2$ now (similarly for $q$). As Schr\"oder's and Julia's equations have proven relevant to many areas (dynamical systems, chaos theory, renormalization groups, etc.), it would be of interest to study the properties of their generalizations in \eqref{1.1}. For more details see Remark \ref{RemSch}.

We further show in Theorem \ref{t3.7} what additional assumptions on $A$ and $\phi$ at the endpoints $x=a,b$ are needed for self-adjoint extensions to be $K$-invariant, an interesting implication of which is Corollary \ref{c3.8}, characterizing the boundary conditions that can describe the Krein--von Neumann extension (assuming a strictly positive minimal operator). We illustrate these results by multiple explicit examples in Section \ref{sub3.1} which yield nontrivial $K$-invariant operators. For instance, the minimal (and maximal) operators associated with the Schr\"odinger differential expression
\begin{align}
\tau=-\frac{d^2}{dx^2}+\dfrac{\g}{\big(1-e^{-\mu^{1/2}x}\big)^{2}}+\dfrac{\mu}{4},\quad 
\g\in(-\mu/4,\infty),\ \mu\in(0,\infty),\ x\in(0,\infty),
\end{align}
are shown to be $K$-invariant where $(Kf)(x)=A_{c,\mu}(x)f(\phi_{c,\mu}(x))$ with
\begin{align}
A_{c,\mu}(x)=\big[1+ce^{-\mu^{1/2}x}\big]^{1/2},\ \phi_{c,\mu}(x)=-\mu^{-1/2}\ln\bigg[\dfrac{(1+c)e^{-\mu^{1/2}x}}{1+ce^{-\mu^{1/2}x}}\bigg],\quad c,\mu\in(0,\infty),\ x\in(0,\infty).
\end{align}

\section{Abstract Framework}\label{s2}

\begin{definition}
Let $K\in\cB(\cH)$ $($the space of bounded operators in $\mathcal{H}$$)$ be a boundedly invertible operator and $T$ be a densely defined and closable operator in a Hilbert space $\cH$. We say that $T$ is \emph{$K$-invariant} if
\begin{equation} \label{eq:conjinv}
K^*TK=T.
\end{equation} 
This means that $\cD(T)=\cD(K^*TK)$ and that $K^*TKf=Tf$ for every $f\in\cD(T)=\cD(K^*TK)$. 
\end{definition}
\begin{remark} \label{rem:invariant}
$(i)$ Observe that $\cD(T)=\cD(K^*TK)=\cD(TK)$ implies $\cD(T)=K\cD(T)=K^{-1}\cD(T)$.\\[1mm]
$(ii)$ If $T$ is $K$-invariant, then $T$ is also $K^n$-invariant for all $n\in\bbZ$.\\[1mm]
$(iii)$ If $K$ is a unitary operator, then $T$ being $K$-invariant is equivalent to $T$ and $K$ commuting.
\end{remark}

For the entirety of this section we assume once and for all the following:
\begin{hypothesis}
The operator $K\in\mathcal{B}(\cH)$ is boundedly invertible and the operator $T$ is a densely defined and closable operator in $\cH$.
\end{hypothesis}

\begin{proposition} \label{prop:2.4}
If $T$ is $K$-invariant, then so are $T^*$ and $\overline{T}$. 
\end{proposition}
\begin{proof}
Using that $T$ is $K$-invariant, this follows from $T^*=(K^*TK)^*=(TK)^*K=K^*T^*K$, where the second equality follows from \cite[Satz 2.43b]{Weidmann} and the last equality from \cite[Satz 2.43c]{Weidmann}. A repeated application of this result to $\overline{T}=T^{**}$ shows that $\overline{T}$ is also $K$-invariant.
\end{proof}

\begin{lemma} Assume that $T$ is $K$-invariant and let $\hat{T}\subseteq T^*$ be a restriction of $T^*$. Then $\hat{T}$ is $K$-invariant if and only if $K\cD(\hat{T})=\cD(\hat{T})$. 
\label{lemma:A3}
\end{lemma}
\begin{proof} First note that by Remark \ref{rem:invariant} $(i)$, it is necessary that $K\cD(\hat{T})=\cD(\hat{T})$ for $\hat{T}$ to be $K$-invariant. Now, assume $K\cD(\hat{T})=\cD(\hat{T})$, which implies $\cD(K^*\hat{T}K)=\cD(\hat{T})$. Then, for any $f\in\cD(\hat{T})$, we get
\begin{equation}
K^*\hat{T}K f=K^*T^*K f= T^*f=\hat{T}f,
\end{equation}
where we used that by Proposition \ref{prop:2.4}, the operator $T^*$ is $K$-invariant. This finishes the proof.
\end{proof}

\subsection{\texorpdfstring{$K$}{K}-invariant nonnegative self-adjoint and maximally dissipative extensions}
In this section, we study the $K$-invariance of the nonnegative self-adjoint and maximally dissipative extensions of a given nonnegative symmetric and $K$-invariant operator $S$. Recall that a symmetric operator $S$ is called nonnegative if 
\begin{equation}
\langle f,Sf\rangle\geq 0\quad\forall f\in\cD(S).
\end{equation}
In this case, we will write $S\geq 0$. If, in addition, there exists a positive constant $\varepsilon>0$ such that
\begin{equation}
\langle f,Sf\rangle\geq \varepsilon \|f\|^2\quad\forall f\in\cD(S),
\end{equation}
we call $S$ strictly positive and write $S\geq \varepsilon I$. A celebrated result in the theory of self-adjoint extensions is that among all nonnegative self-adjoint extensions of a given nonnegative symmetric operator $S$, there are two distinct ones, the Friedrichs extension $S_F$ and the Krein--von Neumann extension $S_K$ \cite{Kr47}. They are characterized by the property that any other nonnegative self-adjoint extension $\hat{S}$ of $S$ satisfies
\begin{equation}
0\leq S_K\leq \hat{S}\leq S_F,
\end{equation}
where the partial order $``A_1\leq A_2"$ for two arbitrary nonnegative self-adjoint operators is defined as 
\begin{equation}
A_1\leq A_2\quad \Leftrightarrow\quad \cD(A_1^{1/2})\supseteq\cD(A_2^{1/2})\:\:\mbox{and}\:\: \|A_1^{1/2}f\|\leq\|A_2^{1/2}f\|
\end{equation}
for all $f\in\cD(A_2^{1/2})$. Following the presentation in \cite{AGMST10}, we provide the following useful characterizations of $S_F$ (due to Freudenthal \cite{F36}) and $S_K$ (due to Ando and Nishio \cite{AN70}). 
\begin{proposition} \label{prop:2.6}
Let $S\geq 0$ be a nonnegative symmetric operator. Then $S_F$ and $S_K$ are given by
\begin{align}
S_F:\:\cD(S_F)&=\big\{f\in\cD(S^*)\: |\: \exists (f_j)_{j\in\mathbb{N}}\subset\cD(S) \mbox{ such that } \notag\\&\quad\quad \lim_{n\rightarrow\infty}\|f_n-f\|=0\mbox{ and } \langle (f_n-f_m), S(f_n-f_m)\rangle\overset{n,m\rightarrow\infty}{\longrightarrow}0\big\},\notag\\
S_F&=S^*\upharpoonright_{\cD(S_F)},\\
S_K:\:\cD(S_K)&=\big\{f\in\cD(S^*)\: |\: \exists (f_j)_{j\in\mathbb{N}}\subset\cD(S) \mbox{ such that } \notag\\&\qquad\lim_{n\rightarrow\infty}\|S^*(f_n-f)\|=0\mbox{ and } \langle (f_n-f_m), S(f_n-f_m)\rangle\overset{n,m\rightarrow\infty}{\longrightarrow}0\big\},\notag\\
S_K&=S^*\upharpoonright_{\cD(S_K)}.
\end{align}
\end{proposition}
Using this characterization, we are now able to prove that the Friedrichs and Krein--von Neumann extensions of a given nonnegative and $K$-invariant self-adjoint operator are also always $K$-invariant:

\begin{theorem} \label{thm:A.7}
Let $S\geq 0$ be a nonnegative symmetric operator which is $K$-invariant. Then its Friedrichs and Krein--von Neumann extensions are also $K$-invariant.
\end{theorem}
\begin{proof} First note that by Proposition \ref{prop:2.4}, the adjoint $S^*$ is also $K$-invariant.

Now, suppose $f\in\cD(S_F)$. Let us show that $Kf\in\cD(S_F)$ as well.  By Proposition \ref{prop:2.6}, since $f\in\cD(S_F)$, there exists a sequence $(f_n)\subset \cD(S)$ such that $\|f-f_n\|\rightarrow 0$ as $n\rightarrow\infty$ and $\langle (f_n-f_m),S(f_n-f_m)\rangle\rightarrow 0$ as $n,m\rightarrow\infty$. Define the sequence $(g_n)_{n\in\mathbb{N}}$ with $g_n:=Kf_n\in K\cD(S)=\cD(S)$. Since $K$ is bounded, we have
\begin{equation}
\|g_n-Kf\|=\|K(f_n-f)\|\leq \|K\|\|f_n-f\|\overset{n\rightarrow\infty}{\longrightarrow}0.
\end{equation}
Likewise, due to $K$-invariance of $S$, we get
\begin{align} \label{eq:2.14}
\langle (g_n-g_m),S(g_n-g_m)\rangle&=\langle K(f_n-f_m),SK(f_n-f_m)\rangle\\
&=\langle f_n-f_m,K^*SK(f_n-f_m)\rangle=\langle f_n-f_m,S(f_n-f_m)\rangle\overset{n,m\rightarrow\infty}{\longrightarrow}0,\notag
\end{align}
which shows that $Kf\in\cD(S_F)$. A completely analogous argument shows that if $f\in\cD(S_F)$, then $K^{-1}f\in\cD(S_F)$ as well, implying $K\cD(S_F)=\cD(S_F)$ and thus by Lemma \ref{lemma:A3} that $S_F$ is $K$-invariant.

Using Proposition \ref{prop:2.6} again, the argument to show $\cD(S_K)=K\cD(S_K)$ is very similar: Assume $f\in\cD(S_K)$, which means there exists a sequence $(f_n)_{n\in\mathbb{N}}\subset \cD(S)$ such that $\|S^*(f-f_n)\|\rightarrow 0$ as $n\rightarrow\infty$ and $\langle (f_n-f_m),S(f_n-f_m)\rangle\rightarrow 0$ as $n,m\rightarrow\infty$. Arguing as in \eqref{eq:2.14}, it follows that the sequence $(g_n)_{n\in\mathbb{N}}\subseteq K\cD(S)=\cD(S)$, where $g_n:=K f_n$, satisfies $\langle (g_n-g_m),S(g_n-g_m)\rangle\rightarrow 0$ as $n,m\rightarrow \infty$. It remains to show that $\|S^*(Kf-g_n)\|\rightarrow 0$ as $n\rightarrow\infty$, which follows from the $K$-invariance of $S^*$:
\begin{align}
\|S^*(Kf-g_n)\|&=\|S^*K(f-f_n)\|=\|(K^*)^{-1}K^*S^*K(f-f_n)\|\notag\\
&=\|(K^*)^{-1}S^*(f-f_n)\|\leq \|(K^*)^{-1}\|\|S^*(f-f_n)\|\overset{n \rightarrow \infty}{\longrightarrow}0\;,
\end{align}
which implies $Kf\in\cD(S_K)$. By a completely analogous argument, it follows that if $f\in\cD(S_K)$, then so is $K^{-1}f$ and therefore $K\cD(S_K)=\cD(S_K)$. By Lemma \ref{lemma:A3}, this implies that $S_K$ is $K$-invariant.
\end{proof}

\begin{example}
Let $\cH=L^2(0,\infty)$ and the nonnegative closed symmetric operator $S$ be given by
\begin{equation}
S:\:\cD(S)=\{f\in H^2(0,\infty)\: |\: f(0)=f'(0)=0\},\quad f\mapsto -f''.
\end{equation}
Note that $S$ is nonnegative, but not strictly positive. It can be verified straightforwardly that all nonnegative self-adjoint extensions of $S$ are given by $\{S_\mu\}_{\mu\in [0,\infty]}$, where $S_\mu$ is defined as follows:
\begin{equation}
S_\mu:\:\cD(S_\mu)=\{f\in H^2(0,\infty)\: |\: f'(0)=\mu f(0)\},\quad f\mapsto -f'',
\end{equation}
with the understanding that $``\mu=\infty"$ corresponds to a Dirichlet condition at zero. It can also be verified that the Friedrichs extension $S_F$ of $S$ correspond to $\mu=\infty$, while its Krein--von Neumann extension $S_K$ corresponds to a Neumann condition at zero, that is, $\mu=0$. Thus, $S_F=S_\infty$ and $S_K=S_0$.
Now, for a fixed $\lambda>0$, $\lambda\neq 1$, we define the scaling transformation $K:L^2(0,\infty)\rightarrow L^2(0,\infty)$ via
\begin{equation}
(Kf)(x)=\frac{1}{\sqrt{\lambda}}f(\lambda x), \text{ with adjoint $K^*$ given by }\ (K^*g)(x)=\frac{1}{\lambda^{3/2}}g(x/\lambda).
\end{equation}
By direct calculation, it can be verified that $S$ is $K$-invariant and thus by Theorem \ref{thm:A.7} so are $S_F$ and $S_K$.
Due to Lemma \ref{lemma:A3}, we only need to check whether $K\cD(S_\mu)=\cD(S_\mu)$ in order to verify whether $S_\mu$ is $K$-invariant. But for any $f\in H^2(0,\infty)\setminus \cD(S)$ such that $f'(0)=\mu f(0)$, this means that we need to require $\lambda^{1/2}f'(0)=(Kf)'(0)=\mu (Kf)(0)=\mu \lambda^{-1/2}f(0)$, which is not possible for $\lambda\neq 1$ and $\mu\in(0,\infty)$. For this example, this shows that the Friedrichs and Krein--von Neumann extensions $S_F$ and $S_K$ of $S$ are the only ones that are $K$-invariant.
\end{example}

Theorem \ref{thm:A.7} can be generalized to show that the Friedrichs extension of a $K$-invariant sectorial operator is also $K$-invariant. Recall that a densely defined operator $A$ in a Hilbert space $\cH$ is called \emph{sectorial} if its numerical range is contained in a sector within the open right half-plane \cite[p.\ 280]{Kato},
\begin{equation}
\{\langle \psi,A\psi\rangle\: |\: \psi\in\cD(A),\: \|\psi\|=1\}\subseteq \{ z\in\mathbb{C}\: |\: -\eta\leq \arg(z)\leq \eta\} \text{ for some $\eta\in[0,\pi/2)$.}
\end{equation}
Moreover, $A$ is called \emph{maximally sectorial} if there exists no nontrivial sectorial extension of $A$. One then closes the domain of $A$ with respect to the norm $\|\cdot\|_A$ given by
\begin{equation}
\|\psi\|_A^2:=\|\psi\|^2+\Re\langle \psi,A\psi\rangle
\end{equation}
    to obtain the \emph{form domain} $\cQ(A):=\overline{\cD(A)}^{\|\cdot\|_A}$  of $A$ (cf.\ \cite[VI, \S 3]{Kato}). The adjoint $A_F^*$ of the Friedrichs extension $A_F$ is maximally sectorial and given by
    \begin{equation} \label{eq:Friedom}
A_F^*:\quad\cD(A_F^*)=\cQ(A)\cap\cD(A^*),\quad A_F^*=A^*\upharpoonright_{\cD(A_F^*)},
    \end{equation}
which is a result shown in \cite[Remarks right after Thm.\ 1]{Arli97}. 

\begin{theorem}
If $A$ is a K-invariant sectorial operator, then its Friedrichs extension is also K-invariant.
\end{theorem}
\begin{proof}
Since $A$ is $K$-invariant, by Lemma \ref{lemma:A3} we need to prove that $K\cD(A_F^*)=\cD(A_F^*)$ in order to show that $A_F^*$ is $K$-invariant. An application of Proposition \ref{prop:2.4} will then imply that also $A_F$ is $K$-invariant. Using the injectivity of $K$ and $\eqref{eq:Friedom}$, we have $K\cD(A_F^*)=K(\cD(A^*)\cap\cQ(A))=K\cD(A^*)\cap K\cQ(A)=\cD(A^*)\cap\ K\cQ(A)$, where we used that $K\cD(A^*)=\cD(A^*)$, since $A^*$ is $K$-invariant. It remains to show that $K\cQ(A)=\cQ(A)$ to conclude that $A_F$ is $K$-invariant. Let $f\in\cQ(A)$ and $(f_n)_{n\in\N}\subset\cD(A)$ such that $\|f-f_n\|\rightarrow 0$ as $n\rightarrow\infty$ and $\|f_n-f_m\|_A\rightarrow 0$ as $n,m\rightarrow \infty$. It follows that $Kf\in\cQ(A)$ since $(Kf_n)_{n\in\N}\subset\cD(A)$, $\|Kf-Kf_n\|\leq\|K\|\|f-f_n\|\rightarrow 0$ as $n\rightarrow\infty$, and 
\begin{align}
\|Kf_n-Kf_n\|_A^2&=\|K(f_n-f_m)\|^2+\Re\langle K(f_n-f_m),AK(f_n-f_m)\rangle\notag\\
&\leq \|K\|^2\|f_n-f_m\|^2  +\Re\langle f_n-f_m,K^*AK(f_n-f_m)\rangle\notag\\&= \|K\|^2\|f_n-f_m\|^2  +\Re\langle f_n-f_m,A(f_n-f_m)\rangle,
\end{align}
which goes to zero as $n,m\rightarrow \infty$, since $(f_n)_{n\in\N}$ is Cauchy with respect to $\|\cdot\|_A$. Analogously, it can be shown that if $f\in\cQ(A)$, then so is $K^{-1}f$, which implies $K\cQ(A)=\cQ(A)$. This completes the proof.
\end{proof}

In what follows, we will focus on the situation when $S$ is a strictly positive operator: $S\geq \varepsilon I$ for some $\varepsilon>0$. We begin with following useful lemma: 
\begin{lemma} \label{lemma:A.8}
Let $S$ be a strictly positive symmetric operator which is $K$-invariant. Then $S_F^{-1}$ is $K^*$-invariant. Moreover, $K\ker(S^*)=\ker(S^*)$.
\end{lemma}
\begin{proof}
First note that $\mathcal{H}=\cD(S_F^{-1})=K^*\cD(S_F^{-1})$. Now, using that $K^*S_FK=S_F$ by Theorem \ref{thm:A.7}, we obtain $S_F^{-1}=(K^*S_FK)^{-1}=K^{-1}S_F^{-1}(K^*)^{-1}$. This then implies $S_F^{-1}=KS_F^{-1}K^*$.
For the second assertion, let $\eta\in\ker(S^*)$. Then, $0=S^*\eta=K^*S^*K\eta$, which implies $K\eta\in\ker(S^*)$. Conversely, if $K\eta\in\ker(S^*)$, then $0=K^*S^*K\eta=S^*\eta$ and thus $\eta\in\ker(S^*)$. This finishes the proof.
\end{proof}
Next, we introduce the notion of dissipative and maximally dissipative operators:
\begin{definition}
Let $A$ be a densely defined operator in a Hilbert space $\cH$. Then it is called \emph{dissipative} if $\Im\langle f,Af\rangle\geq 0$
for all $f\in\cD(A)$. If, in addition, there is no nontrivial dissipative extension of $A$, then it is called \emph{maximally dissipative}.
\end{definition}
The following result was shown by Grubb \cite{Gr68} (see also \cite[Thm.\ 7.2.4]{F17}):
\begin{proposition} \label{prop:A.10}
Let $S$ be a strictly positive, closed, symmetric operator. Then all nonnegative self-adjoint/maximally dissipative extensions of $S$ are of the form
\begin{align}
S_B:\:\cD(S_B)&=\cD(S)\dot{+}\{S_F^{-1}Bf+f|f\in\cD(B)\}\dot{+}\{S_F^{-1}\eta|\eta\in \cD(B)^\perp\cap\ker(S^*)\}, \notag\\
S_B&=S^*\upharpoonright_{\cD(S_B)}, \label{eq:2.15}
\end{align}
where $B$ is a nonnegative self-adjoint/maximally dissipative operator in $\overline{\cD(B)}\subseteq\ker(S^*)$.
\end{proposition}
The following result provides the necessary and sufficient condition that the auxiliary operator $B$ describing a self-adjoint/maximally dissipative extension $S_B$ of $S$ has to satisfy in order to ensure that $S_B$ is $K$-invariant as well.
\begin{theorem} \label{thm:A.10}
Let $S$ be a strictly positive, closed, symmetric operator. Then $S_B$ is $K$-invariant if and only if $\cD(B)=\cD(K^*BK)$ and
\begin{equation}  \label{eq:conjcond}
P_{\overline{\cD(B)}}K^*BK\upharpoonright_{\mathcal{D}(B)}=B,
\end{equation}
where $P_{\overline{\cD(B)}}$ denotes the orthogonal projection onto $\overline{\cD(B)}$.
\end{theorem}
\begin{proof}
First assume that $\cD(B)=\cD(K^*BK)$ and \eqref{eq:conjcond} is satisfied. By Lemma \ref{lemma:A3}, we must show that $K\cD(S_B)=\cD(S_B)$, or, equivalently, $\cD(S_B)=K^{-1}\cD(S_B)$.
Let $\psi\in\cD(S_B)$, that is, there exist unique $f_0\in\cD(S), f\in\cD(B)$, and $\eta\in\cD(B)^\perp\cap\ker(S^*)$ such that $\psi=f_0+S^{-1}_FBf+f+S_F^{-1}\eta.$
Then
\begin{align}
K^{-1} \psi\hspace{-1pt}&=K^{-1} f_0+K^{-1} S^{-1}_FBf+K^{-1}f+K^{-1} S_F^{-1}\eta \notag\\&= K^{-1} f_0+S_F^{-1}K^*Bf+K^{-1}f+S_F^{-1}K^*\eta \\
&= K^{-1} f_0+S_F^{-1}P_{{\cD(B)}^\perp}K^*Bf+S_F^{-1}P_{\overline{\cD(B)}}K^*BKK^{-1}f+K^{-1}f+S_F^{-1}K^*\eta, \notag
\label{eq:2.18}
\end{align}
where $P_{{\cD(B)}^\perp}$ denotes the orthogonal projection onto $\cD(B)^\perp$ and we have used that by Lemma \ref{lemma:A.8}, $S_F^{-1}$ is $K^*$-invariant and thus $K^{-1}S_F^{-1}=S_F^{-1}K^*$. Now, since by assumption, $K^{-1}f\in\cD(B)$ and \eqref{eq:conjcond} holds, we can simplify 
\begin{equation}
S_F^{-1}P_{\overline{\cD(B)}}K^*BKK^{-1}f=S_F^{-1}BK^{-1}f,
\end{equation}
and therefore obtain
\begin{equation} \label{eq:2.22}
K^{-1}\psi=K^{-1} f_0+S_F^{-1}P_{{\cD(B)}^\perp}K^*Bf+S_F^{-1}BK^{-1}f+K^{-1}f+S_F^{-1}K^*\eta.
\end{equation}

Let us now argue that $K^{-1}\psi\in\cD(S_B)$, that is, we need to show there exist $\tilde{f}_0\in\cD(S)$, $\tilde{f}\in\cD(B)$, and $\tilde{\eta}\in\ker(S^*)\cap\cD(B)^\perp$ such that $K^{-1}\psi=\tilde{f}_0+S_F^{-1}B\tilde{f}+\tilde{f}+S_F^{-1}\tilde{\eta}$.

Since we assumed $S$ to be $K$-invariant, we have $K^{-1} f_0\in\cD(S)$. Moreover, note that if $\eta\in\cD(B)^\perp\cap\ker(S^*)$, then for any $\phi\in\cD(B)$:
\begin{equation}
\langle \phi,K^*\eta\rangle=\langle K\phi,\eta\rangle=0,
\end{equation}
since $K\phi\in\cD(B)$ by assumption. Thus, $K^*\eta\in\cD(B)^\perp$ and, trivially, we also have $P_{{\cD(B)}^\perp}K^*Bf\in\cD(B)^\perp$. Next, decompose $K^*\eta=\kappa_1+\kappa_2$ and $P_{{\cD(B)}^\perp}K^*Bf=\sigma_1+\sigma_2$, where $\kappa_1, \sigma_1\in\ran(S)$ and $\kappa_2, \sigma_2\in\ker(S^*)$. (Note that since $S$ is assumed to be strictly positive and closed, its range $\ran(S)$ is a closed subspace.) We have $\kappa_1, \sigma_1\in\ran(S)=\ker(S^*)^\perp\subseteq\cD(B)^\perp$. Consequently, since $\cD(B)^\perp$ is a linear space, we conclude that $\kappa_2,\sigma_2\in\cD(B)^\perp$ as well and therefore $\kappa_2,\sigma_2\in\ker(S^*)\cap\cD(B)^\perp$. Thus, Equation \eqref{eq:2.22} can be rewritten as 
\begin{align}
K^{-1}\psi&=K^{-1} f_0+S_F^{-1}BK^{-1}f+K^{-1}f+S_F^{-1}P_{{\cD(B)}^\perp}K^*Bf+S_F^{-1}K^*\eta\notag\\
&= (K f_0+S_F^{-1}(\kappa_1+\sigma_1))+S_F^{-1}BK^{-1}f+K^{-1}f+S_F^{-1}(\kappa_2+\sigma_2),
\end{align}
where $(K^{-1} f_0+S_F^{-1}(\kappa_1+\sigma_1))\in\cD(S)$, which follows from $K^{-1}f_0\in\cD(S)$ and $S_F^{-1}(\kappa_1+\sigma_1)\in\cD(S)$ since $\kappa_1, \sigma_1\in\ran(S)$. Moreover, $K^{-1}f\in\cD(B)$ and $(\kappa_2+\sigma_2)\in\ker(S^*)\cap\cD(B)^\perp$. This implies that if $\psi\in\cD(S_B)$, then $K^{-1}\psi\in\cD(S_B)$, that is, the inclusion $\cD(S_B)\subseteq K\cD(S_B)$. The other inclusion $K\cD(S_B)\subseteq\cD(S_B)$ follows from a completely analogous argument.

Next, let us show that $\cD(B)=\cD(K^*BK)$ is necessary for $S_B$ to be $K$-invariant. Assume this is not the case. Then there either exists an $f\in\cD(B)$ such that $Kf\notin \cD(B)$ or $K^{-1}f\notin \cD(B)$. We will lead the case $K^{-1}f\notin\cD(B)$ to a contradiction, while the case $Kf\notin\cD(B)$ can be treated analogously. Since $f\in\cD(B)$, this implies that $\psi:=S_F^{-1}Bf+f\in\cD(S_B)$. For $\cD(S_B)=K\cD(S_B)$ to be true, we therefore need that $K^{-1}\psi=K^{-1}S_F^{-1}Bf+K^{-1}f=S_F^{-1}K^*Bf+K^{-1}f\in\cD(S_B)$. If this is true, then there exist unique $\tilde{f}_0\in\cD(S)$, $\tilde{f}\in\cD(B)$, and $\tilde{\eta}\in\ker(S)\cap\cD(B)^\perp$ such that 
\begin{equation}
K^{-1}\psi=K^{-1}S_F^{-1}Bf+K^{-1}f=S_F^{-1}K^*Bf+K^{-1}f=\tilde{f}_0+S_F^{-1}B\tilde{f}+\tilde{f}+S_F^{-1}\tilde{\eta},
\end{equation}
which can be rewritten as 
\begin{equation}
(K^{-1}f-\tilde{f})=\tilde{f}_0+S_F^{-1}(B\tilde{f}+\tilde{\eta}-K^*Bf).
\end{equation}
Now, note that the left-hand side of this is equation is an element of $\ker(S^*)$, where we used that $K^{-1}f\in\ker(S^*)$ by Lemma \ref{lemma:A.8}. However, the right-hand side is an element of $\cD(S_F)$, from which we conclude both sides must be $0$ since $\cD(S_F)\cap\ker(S^*)=\{0\}$. Thus, $K^{-1}f=\tilde{f}\in\cD(B)$, which contradicts $K^{-1}f\notin\cD(B)$. Hence we conclude that $\cD(B)=\cD(K^*BK)$ is necessary for $S_B$ to be $K$-invariant.

Now, assume $\cD(B)=\cD(K^*BK)$, but \eqref{eq:conjcond} is not satisfied, that is, there exists $f\in\cD(B)$ such that 
\begin{equation} \label{eq:2.28}
P_{\overline{\cD(B)}}K^*BKK^{-1}f\neq BK^{-1}f.
\end{equation}
Again, we have $\psi:=S_F^{-1}Bf+f\in\cD(S_B)$ and for $\cD(S_B)=K\cD(S_B)$ to be true, it must hold that
\begin{equation} \label{eq:2.29}
K^{-1}\psi=S_F^{-1}K^*Bf+K^{-1}f=\tilde{f}_0+S_F^{-1}B\tilde{f}+\tilde{f}+S_F^{-1}\tilde{\eta}
\end{equation}
for some $\tilde{f}_0\in\cD(S)$, $\tilde{f}\in\cD(B)$, and $\tilde{\eta}\in\ker(S^*)\cap\cD(B)^\perp$. Arguing exactly as before, it follows that $K^{-1}f=\tilde{f}$ and therefore, Equation \eqref{eq:2.29} can be rewritten as
\begin{equation}
S_F^{-1}\left(P_{\overline{\cD(B)}}K^*BKK^{-1}f-BK^{-1}f\right)=\tilde{f}_0+S_F^{-1}\left(\tilde{\eta}-P_{{\cD(B)}^\perp}K^*BKK^{-1}f\right).
\end{equation}
Again, both sides of this equation are linearly independent, since the left-hand side is an element of $S_F^{-1}\overline{\cD(B)}$, while the right-hand side lies in 
\begin{equation}
S_F^{-1}\left(\ran(S)\oplus(\ker(S^*)\cap\cD(B)^\perp)\right).
\end{equation}
Consequently, both sides must be equal to $0$. However, by \eqref{eq:2.28}, using that $S_F^{-1}$ is injective, the left-hand side is not zero, which is a contradiction. This finishes the proof.
\end{proof}
If we further require that the map $K$ be unitary, then the necessary and sufficient conditions for a nonnegative self-adjoint/maximally dissipative extension $S_B$ of a given $K$-invariant operator $S$ to also be $K$-invariant can be simplified further:
\begin{corollary}
In addition to the assumptions of Theorem \ref{thm:A.10}, assume that $K$ is a unitary operator. Then $S_B$ is $K$-invariant if and only if $\cD(B)=\cD(K^*BK)$ and $K^*BK=B$.
\end{corollary}
\begin{proof}
Using that by Theorem \ref{thm:A.10}, the extension $S_B$ is $K$-invariant if and only if $\cD(B)=\cD(K^*BK)$ and Condition \eqref{eq:conjcond} is satisfied, the corollary follows if we can show that 
\begin{equation}
P_{\overline{\cD(B)}}K^*BKf=K^*BKf
\end{equation}
for all $f\in \cD(B)$. Since $\cD(B)=\cD(K^*BK)$ and $B$ is a self-adjoint/maximally dissipative operator in $\overline{\cD(B)}$, we know that $BKf\in\overline{\cD(B)}$. Now, let $(\eta_n)\subseteq\cD(B)$ be a sequence such that $\eta_n\rightarrow BKf$. Since $K$ is unitary, $K^*=K^{-1}$, this implies $\cD(B)=\cD(K^*BK)=K^{-1}\cD(B)=K^*\cD(B)$. Hence, $K^*\eta_n\in\cD(B)$ for every $n$ and moreover $K^*\eta_n\rightarrow K^*BKf$, which therefore has to be an element of $\overline{\cD(B)}$.
\end{proof} 

Next, given a strictly positive $K$-invariant symmetric operator, we construct a class of nonnegative self-adjoint extensions that are guaranteed to also be $K$-invariant. They are characterized by the choice $B\equiv 0$, however, we can choose different closed subspaces $\mathfrak{M}$ of $\ker(S^*)$ on which $B\equiv 0$ is defined:  
\begin{theorem} \label{thm:2.13}
Let $S$ be a strictly positive, closed, symmetric operator which is $K$-invariant. Let $\mathfrak{M}$ be a closed subspace of $\ker(S^*)$ such that $K^{-1}\mathfrak{M}=\mathfrak{M}$. Then the operator $S_\mathfrak{M}$ given by
\begin{equation}
S_\mathfrak{M}:\:\cD(S_\mathfrak{M})=\cD(S)\dot{+}\mathfrak{M}\dot{+}\{S_F^{-1}\eta\: |\: \eta\in\mathfrak{M}^\perp\cap\ker(S^*)\},\quad S_\mathfrak{M}=S^*\upharpoonright_{\cD(S_\mathfrak{M})},
\end{equation}
is also $K$-invariant.
\end{theorem}
\begin{proof}
This immediately follows from Theorem \ref{thm:A.10} using that $S_\mathfrak{M}$ corresponds to the choice \begin{equation}
B:\quad\cD(B)=\overline{\cD(B)}=\mathfrak{M},\quad f\mapsto 0,
\end{equation}
that is, $B$ is the zero operator on $\cD(B)=\mathfrak{M}$. Since $\mathcal{D}(K^*BK)=K^{-1}\cD(B)=K^{-1}\mathfrak{M}=\mathfrak{M}=\mathcal{D}(B)$ by assumption and Condition \eqref{eq:conjcond} is always satisfied for $B\equiv 0$, this shows the corollary.
\end{proof}
\begin{remark}\label{rem2.12}
By Theorem \ref{thm:A.7}, we already know that the Krein--von Neumann extension $S$ is always $K$-invariant. Nevertheless, for the strictly positive case $S\geq \varepsilon I$, Theorem \ref{thm:2.13} provides an alternative proof of this fact using Lemma \ref{lemma:A.8} together with the choice $\mathfrak{M}=\ker(S^*)$, which corresponds to the Krein--von Neumann extension.
Moreover, assume that $\ker(S^*)$ is finite-dimensional with $\dim\ker(S^*)>1$. Then there exists at least one nontrivial proper subspace $\mathfrak{M}$ of $\ker(S^*)$, spanned by one or more eigenvectors of $K^{-1}$, such that $K^{-1}\mathfrak{M}=\mathfrak{M}$. Therefore there also exists at least one additional nonnegative $K$-invariant self-adjoint extension of $S$ that is distinct from $S_F$ and $S_K$. See Example \ref{directsumexample}.
\end{remark}

In what follows, we restrict ourselves to nonnegative self-adjoint extensions and focus on the case when the operator $S$ has finite defect index, $\dim(\ker(S^*))<\infty$. In this case, we trivially have $\cD(B)=\overline{\cD(B)}$ for any auxiliary operator $B$ defined on the finite-dimensional space $\cD(B)\subseteq\ker(S^*)$.
By Theorem \ref{thm:A.10}, for $S_B$ to be $K$-invariant, it is necessary that $\cD(B)=K\cD(B)$. This implies that $\tilde{K}:=K\upharpoonright_{\cD(B)}$ is a linear mapping from $\cD(B)$ to $\cD(B)$ and therefore unitarily equivalent to a square-matrix. For any $\zeta\in\sigma(\tilde{K})$, we introduce its root space $\mathcal{R}(\zeta,\tilde{K})$:
\begin{equation}
\mathcal{R}(\zeta,\tilde{K})=\left\{\eta\in\cD(B)\: \Big|\: \exists n\in\bbN \mbox{ such that } (\tilde{K}-\zeta)^n\eta=0\right\}.
\end{equation}
Recall that $\cD(B)$ is given by the direct sum of all root spaces corresponding to all eigenvalues of $\tilde{K}$:
\begin{equation}
\cD(B)=\mbox{span}\left\{\mathcal{R}(\zeta,\tilde{K})\: \Big|\: \zeta\in\sigma(\tilde{K})\right\}.
\end{equation}
We also introduce the subspace $\mathcal{C}$ of $\cD(B)$ given by the direct sum of root spaces corresponding to eigenvalues $\zeta$ with $|\zeta|\neq 1$:
\begin{equation} \label{eq:2.42a}
\cC:=\mbox{span}\left\{\mathcal{R}(\zeta,\tilde{K})\: \Big|\: \zeta\in\sigma(\tilde{K}),\: |\zeta|\neq 1\right\}.
\end{equation}
\begin{theorem} \label{thm:2.17}
Let $S$ be a strictly positive, closed, symmetric operator which is $K$-invariant. Moreover, assume that $\dim(\ker(S^*))<\infty$. Then for a nonnegative self-adjoint extension $S_B$ of $S$  to be $K$-invariant it is necessary that $\cC\subseteq\ker{B}$.
\end{theorem}
\begin{proof} Let $\eta\in\cD(B)$ be an eigenvector of $\tilde{K}$ corresponding to an eigenvalue $\zeta$ with $|\zeta|\neq 1$. By Condition \eqref{eq:conjcond}, for $S_B$ to be $K$-invariant, it is necessary that
\begin{equation}
P_{\cD(B)}K^*BK\upharpoonright_{\cD(B)}=B,
\end{equation}
and thus, in particular, 
\begin{equation} \label{eq:2.39}
\langle \psi,P_{\cD(B)}K^*BK\psi\rangle=\langle \tilde{K}\psi,B\tilde{K}\psi\rangle=\langle\psi,B\psi\rangle,
\end{equation}
for every $\psi\in\cD(B)$. Choosing $\psi=\eta$ then yields the condition
\begin{align}
|\zeta|^2\|B^{1/2}\eta\|^2=|\zeta|^2\langle \eta,B\eta\rangle=\langle \tilde{K}\eta,B\tilde{K}\eta\rangle=\langle \eta,B\eta\rangle=\|B^{1/2}\eta\|^2,
\end{align}
or equivalently
\begin{equation}
(1-|\zeta|^2)\|B^{1/2}\eta\|^2=0.
\end{equation}

Since $|\zeta|\neq 1$, it follows that $\|B^{1/2}\eta\|=0$ and therefore $\eta\in\ker(B^{1/2})=\ker(B)$. Next, assume that $\tilde{\eta}\in\mathcal{R}(\zeta,\tilde{K})$ is a root vector such that $(\tilde{K}-\zeta)\tilde{\eta}=\eta$ and therefore $\tilde{K}\tilde{\eta}=\zeta\tilde{\eta}+\eta$. Plugging this into \eqref{eq:2.39} and using that $\eta\in\ker(B)$ yields the condition
\begin{equation}
\langle \tilde{K}\tilde{\eta},B\tilde{K}\tilde{\eta}\rangle=\langle \zeta\tilde{\eta}+\eta,B(\zeta\tilde{\eta}+\eta)\rangle=\|B^{1/2}(\zeta\tilde{\eta}+\eta)\|^2=|\zeta|^2\|B^{1/2}\tilde{\eta}\|^2=\|B^{1/2}\tilde{\eta}\|^2,
\end{equation}
which implies by a similar reasoning that $\tilde{\eta}\in\ker(B)$. Arguing analogously for the subsequent members of the Jordan chain spanning $\mathcal{R}(\zeta,\tilde{K})$ shows that $\mathcal{R}(\zeta,\tilde{K})\subseteq\ker(B)$. This finishes the proof.
\end{proof}
For the special case $\dim(\ker(S^*))=1$, we have the following result:
\begin{theorem}\label{t2.17}
Let $S$ be as in Theorem \ref{thm:2.13} and assume that $\dim(\ker(S^*))=1$. Let $\eta$ be a normalized vector which spans $\ker(S^*)$. Then there are the following two cases:\\[1mm] 
$(i)$  $K\eta=\zeta\eta$ with $|\zeta|\neq 1$. Then the Friedrichs extension $S_F$ and the Krein--von Neumann extension $S_K$ of $S$ are the only two $K$-invariant maximally dissipative extensions  of $S$.\\[1mm] 
$(ii)$ $K\eta=\zeta\eta$ with $|\zeta|=1$. Then all maximally dissipative extensions of $S$ are $K$-invariant.
\end{theorem}
\begin{proof}
By Theorem \ref{thm:2.13}, the Friedrichs extension $S_F$ of $S$ is $K$-invariant. Hence, we consider only the case $\cD(B)=\ker(S^*)$ from now on.
Since $\ker(S^*)=\mbox{span}\{\eta\}$ and $K\ker(S^*)=\ker(S^*)$ by Lemma \ref{lemma:A.8}, this means that $\eta$ is an eigenvector of $K$. Let $\zeta\in\mathbb{C}$ denote the corresponding eigenvalue: $K\eta=\zeta\eta$. Moreover, since $\cD(B)$ is one-dimensional, $\eta$ is also an eigenvector with eigenvalue $b$ such that $\Im b\geq 0$. Then Condition \eqref{eq:conjcond} takes the following form:
\begin{equation}
P K^*BK\eta=B\eta,
\end{equation}
where $P=\eta\langle \eta,\cdot\rangle$ is the orthogonal projection onto $\mbox{span}\{\eta\}$, where for convenience, we assume that $\eta$ is normalized. This can be rewritten as
\begin{equation}
PK^*BK\eta=\eta\langle\eta,K^*BK\eta\rangle=\eta\langle K\eta, BK\eta\rangle=|\zeta|^2 b\eta=b\eta=B\eta\,
\end{equation}
or equivalently 
\begin{equation}
(1-|\zeta|^2)b\eta=0.
\end{equation}
Thus, in Case $(i)$, that is, if $|\zeta|\neq 1$, this will only be satisfied if $b=0$, which corresponds to the Krein--von Neumann extension of $S$. On the other hand, in Case $(ii)$ with $|\zeta|=1$, this will be satisfied for any $b$ with $\Im b\geq 0$ corresponding to any maximally dissipative extension of $S$ (cf.\ Proposition \ref{prop:A.10}).
\end{proof}

%%%%%%%%%%%%%%%%%%%%%%%%
%%%%%%%%%%%%%%%%%%%%%%%% 
\section{Application to Sturm--Liouville operators}\label{s3}
%%%%%%%%%%%%%%%%%%%%%%%%
%%%%%%%%%%%%%%%%%%%%%%%%

We now illustrate our previous results via examples involving Sturm--Liouville operators. For general theory, we refer to \cite{GNZ24,Ze05} which contain very detailed lists of references (see also \cite[Sect. 2]{FS25}).
Throughout this section we make the following assumptions, though we note that complex-valued coefficient functions are admissible until considering symmetric differential expressions:

%%%%%%%%%%%%%
\begin{hypothesis} \label{h2.1}
Let $(a,b) \subseteq \bbR$ and suppose that $p,q,r$ are $($Lebesgue\,$)$ measurable functions on $(a,b)$ 
such that $r,p>0$ a.e. on $(a,b)$, $q$ is real-valued a.e. on $(a,b)$, and $r,1/p,q\in\Ll$.
\end{hypothesis}
%%%%%%%%%%%%%
Given Hypothesis \ref{h2.1}, we study Sturm--Liouville operators, $T$, in $\Lr$ associated with the general, 
three-coefficient differential expression
\begin{align}\lb{2.1}
\tau=\f{1}{r(x)}\left[-\f{d}{dx}p(x)\f{d}{dx} + q(x)\right] \, \text{ for a.e.~$x\in(a,b) \subseteq \R$.} 
\end{align}

As usual, the minimal and maximal operators are now defined as follows:

%%%%%%%%%%%%%%
\begin{definition} 
Assume Hypothesis \ref{h2.1}. Given $\tau$ in \eqref{2.1}, the \textit{maximal operator} $T_{max}$ and \textit{preminimal operator} $T_{min,0} $ in $\Lr$ associated with $\tau$ are defined by
\begin{align}
&T_{max} f = \tau f,  
\quad f \in \cD(T_{max})=\big\{g\in\Lr \, \big| \,g,g^{[1]}\in\ACl;  \tau g\in\Lr\big\},\\
&T_{min,0}  f = \tau f,   
\quad f \in \cD (T_{min,0})=\big\{g\in\cD(T_{max}) \, \big| \, \supp \, (g)\subset(a,b) \text{ is compact}\big\},
\end{align}
with the Wronskian $($and quasi-derivative$)$ of $f,g\in\ACl$ defined by
\begin{equation}
W(f,g)(x) = f(x)g^{[1]}(x) - f^{[1]}(x)g(x), \quad
y^{[1]}(x) = p(x) y'(x), \quad x \in (a,b).
\end{equation}
$T_{min,0} $ is symmetric and thus closable, so that one then defines $T_{min}$ 
as the closure of $T_{min,0} $.
\end{definition}
%%%%%%%%%%%%%%

It is well known that
$(T_{min,0})^* = T_{max},$
and hence $T_{max}$ is closed.
Moreover, $T_{min,0} $ is essentially self-adjoint if and only if $T_{max}$ is symmetric, and then 
$\ol{T_{min,0} }=T_{min}=T_{max}$.

In what follows, we will use the notation $[a,b]$ noting that whenever $a$ or $b$ is infinite, the corresponding interval is understood as $(-\infty,b]$ or $[a,\infty)$, respectively (or $(-\infty,\infty)$ if both are infinite) in order to alleviate writing each case separately. 
We now define the operator $K$ we consider.

\begin{lemma}\label{lemmabdinv}
Assume $A,A^{[1]}\in AC([a,b])$ $($understood as $AC_{loc}$ near an infinite endpoint$)$ and $\phi\in C^2([a,b])$ satisfy $\phi'(x)\neq0$ for $x\in[a,b]$, $\phi(d)=d$ for $d\in\{a,b\}$, as well as $\sup_{x\in(a,b)}\big[A^2/\phi'\big](x)$, $\sup_{x\in(a,b)}\big[\phi'/A^2 \big](x)$ exist $($note that this implies that $\phi'(x)>0$ for $x\in[a,b]$ and $A$ is nonzero$)$. In addition, assume $r(x)=C r\big(\phi^{-1}(x)\big)$ for some $C>0$ satisfies Hypothesis \ref{h2.1}.  Then the operator $K:\Lr\to\Lr$ is bounded and boundedly invertible where
\begin{equation}\label{kdef}
(Kf)(x)=A(x) f(\phi(x)).
\end{equation}
Furthermore, the adjoint of $K$ is given by
\begin{equation}
\big(K^* f\big)(x)=\frac{A\big(\phi^{-1}(x)\big) }{C}\big(\phi^{-1})'(x) f\big(\phi^{-1}(x)\big)=\frac{A\big(\phi^{-1}(x)\big)}{C\phi'\big(\phi^{-1}(x)\big)} f\big(\phi^{-1}(x)\big).
\end{equation}
\end{lemma}
\begin{proof}
The fact that $K$ is bounded under the given assumptions is immediate from the fact that, letting $M=C^{-1}\sup_{x\in(a,b)} \big[A^2/\phi'\big](x)$ and abbreviating $\Lr$ by $L_r^2$, 
\begin{align}
\norm{Kf}^2_{L^2_r}&=\int_a^b A^2 (x)|f(\phi(x))|^2\, r(x)dx=\int_a^b \frac{A^2\big(\phi^{-1}(x)\big)}{C}\big(\phi^{-1}\big)'(x) |f(x)|^2\, r(x)dx \leq M\norm{f}^2_{L^2_r}.
\end{align}
Moreover, the inverse of $K$ is given by $\big(K^{-1}f\big)(x)=f\big(\phi^{-1}(x)\big)/A\big(\phi^{-1}(x)\big)$ with norm (which is finite by the assumption $\sup_{x\in(a,b)}\big[\phi'/A^2 \big](x)$ exists)
\begin{align}
\norm{K^{-1}f}^2_{L^2_r}=\int_a^b  \frac{\big|f\big(\phi^{-1}(x)\big)\big|^2}{A^2\big(\phi^{-1}(x)\big)}\, r(x)dx=\int_a^b \frac{C \phi'(x)}{A^2(x) }|f(x)|^2r(x)dx.
\end{align}
The formula for the adjoint now follows from the computation
\begin{align}
\langle g,Kf\rangle_{L_r^2}&=\int_a^b A(x)\overline{g(x)}f(\phi(x))\, r(x)dx=\int_a^b \frac{A\big(\phi^{-1}(x)\big)}{C}\big(\phi^{-1}\big)'(x)\overline{g\big(\phi^{-1}(x)\big)}f(x)\, r(x)dx, 
\end{align}
where we have used that $r(x)=C r\big(\phi^{-1}(x)\big)$ for some $C>0$.
\end{proof}

We now state sufficient requirements on the coefficient functions $p,q,r$ for our main theorem regarding $K$-invariance of Sturm--Liouville operators with $K$ as above.

\begin{hypothesis}\label{hypeqs2}
In addition to Hypothesis \ref{h2.1}, assume that the following hold:
\begin{align}
\begin{split}\label{requirevar}
r(x)&=Cr\big(\phi^{-1}(x)\big),\quad p(x)=\big[A\big(\phi^{-1}(x)\big)\big]^2\phi'\big(\phi^{-1}(x)\big)p\big(\phi^{-1}(x)\big),\\
q(x)&=\frac{A\big(\phi^{-1}(x)\big)}{\phi'\big(\phi^{-1}(x)\big)}\big\{A\big(\phi^{-1}(x)\big) q \big(\phi^{-1}(x)\big)-\big(A^{[1]}\big)'\big(\phi^{-1}(x)\big)\big\},
\end{split}
\end{align}
where $A,A^{[1]}\in AC([a,b])$ $($understood as $AC_{loc}$ near an infinite endpoint$)$ and $\phi\in C^2([a,b])$ satisfy $\phi'(x)\neq0$ for $x\in[a,b]$, $\phi(d)=d$ for $d\in\{a,b\}$, and $\sup_{x\in(a,b)}\big[A^2/\phi'\big](x)$, $\sup_{x\in(a,b)}\big[\phi'/A^2 \big](x)$ exist.
\end{hypothesis}

\begin{theorem}\label{thm3.5}
Assume Hypothesis \ref{hypeqs2} and let $K$ be defined via \eqref{kdef}. Then $T_{min,0}$, $T_{min}$, and $T_{max}$ are $K$-invariant.
\end{theorem}
\begin{proof}
Notice that $K$ is boundedly invertible by Lemma \ref{lemmabdinv}. Moreover, assuming Hypothesis \ref{hypeqs2} holds, a straightforward calculation now yields for $f,f^{[1]}\in AC_{loc}((a,b))$ (using $r(x)=C r\big(\phi^{-1}(x)\big)$),
\begin{align}\label{longeq}
\big(K^*\tau Kf\big)(x)&=\frac{1}{r(x)}\bigg[-\frac{d}{dx}p\big(\phi^{-1}(x)\big)\big[A\big(\phi^{-1}(x)\big)\big]^2\phi'\big(\phi^{-1}(x)\big)\frac{d}{dx}f(x)  \\
&\quad +\frac{A\big(\phi^{-1}(x)\big)}{\phi'\big(\phi^{-1}(x)\big)}\big\{A\big(\phi^{-1}(x)\big) q \big(\phi^{-1}(x)\big) - \big(A^{[1]}\big)'\big(\phi^{-1}(x)\big)\big\} f(x)\bigg],\notag
\end{align}
since comparing the actions of $\tau$ in \eqref{2.1} and $K^*\tau K$ in \eqref{longeq} yields $K^*\tau Kf=\tau f$ by \eqref{requirevar}. 
Regarding the needed domain equality, we first show that if $g\in\mathcal{D}(T_{min,0})$, then so is $Kg$. First of all, notice that since $K$ is bounded, one has $Kg\in L^2_r$ if $g\in L^2_r$, whereas the definition of $K$ (along with the assumptions on $\phi$) implies that if $\supp(g)\subset(a,b)$ is compact, then so is $\supp(Kg)$. Moreover, the assumptions on $\phi$ and $A$ in Hypothesis \ref{hypeqs2} guarantee that $Kg\in AC_{loc}$ if $g$ is and $(Kg)^{[1]}\in AC_{loc}$ if $g^{[1]}$ is by direct calculation. Also, if $\tau g\in L^2_r$ we have that $\tau Kg \in L^2_r$ since, under the assumptions of Hypothesis \ref{hypeqs2},
\begin{align}
\norm{\tau Kg}&=\norm{\big(K^*\big)^{-1} K^{*}\tau Kg}\leq\norm{\big(K^*\big)^{-1} }_{}\norm{\tau g}=\norm{\big(K^{-1}\big)^{*} }_{}\norm{\tau g} =\norm{K^{-1} }_{}\norm{\tau g}<\infty,
\end{align}
where we have used \eqref{longeq}.

On the other hand, since $K$ consists of composition with a $C^2$ function with the endpoints $x=a,b$ as fixed points along with multiplication by nonzero $A$ such that $A,A^{[1]}\in AC([a,b])$, if $Kg\in\mathcal{D}(T_{min,0})$ then $g$ must also satisfy all the properties to be in $\mathcal{D}(T_{min,0})$. This follows from the fact that $K$ cannot improve the regularity of $g$ or make non-compact support compact, and
\begin{align}
\norm{\tau g}&=\norm{K^{*}\tau Kg}\leq\norm{K^* }_{}\norm{\tau K g}.<\infty.
\end{align}
This proves that $T_{min,0}$ is $K$-invariant, while Proposition \ref{prop:2.4} yields that $T_{min}$ and $T_{max}$ are as well.
\end{proof}

A few remarks are now in order regarding \eqref{requirevar}.
Notice that if $A$ is constant, in order to be able to choose $p(x)=1$ in \eqref{requirevar} one must have $\phi'\big(\phi^{-1}(x)\big)=A^{-2}$, that is, $(\phi^{-1}\big)'(x)=A^2$. This implies $\phi^{-1}(x)=A^2 x+B$ for some constant $B$ so that $\phi(x)=A^{-2}(x-B)$. Moreover, since the endpoints $x=a,b$ are the fixed points of $\phi$, we must have $A=\pm1$ and $B=0$ if either endpoint $x=a$ or $x=b$ is finite, that is, only $\phi(x)=x$ is admissible! On the other hand, if $(a,b)=(-\infty,\infty)$, the equation for the potential is now $q(x)=A^{4}q(A^2x+B)$ which has the (Bessel potential) solution $q(x)=c A^{4}(x+B/(A^2-1))^{-2}$, $A\in\bbR\backslash\{-1,0,1\}$, $B,c\in\bbR$, though this potential is not integrable near $x=-B/(A^2-1)$. Finally, the cases $A=\pm1$ are solved by constant $q$.

Thus one must consider variable $A(x)$ when considering operators such that $p(x)$ is constant to avoid reducing to the above trivial cases.
In particular, choosing $C=r(x)=1=p(x)$, that is, considering a Schr\"odinger operator, the requirements on the coefficient functions in Hypothesis \ref{hypeqs2} become
\begin{equation}\label{schrodrequire}
[A(x)]^2\phi'(x)=1, \quad q(x)=\big[A\big(\phi^{-1}(x)\big)\big]^3\big\{A\big(\phi^{-1}(x)\big) q \big(\phi^{-1}(x)\big) -A''\big(\phi^{-1}(x)\big)\big\},
\end{equation}
yielding an interesting and nontrivial $K$-invariance for Schr\"odinger operators (see Example \ref{exampleSchrod}).

\begin{remark}\label{RemSch}
Throughout this remark we assume $A$ in Hypothesis \ref{hypeqs2} is constant, reducing \eqref{requirevar} to
\begin{align}\label{require}
r(x)=Cr\big(\phi^{-1}(x)\big),\, p(x)=A^2 \phi'\big(\phi^{-1}(x)\big) p\big(\phi^{-1}(x)\big),\, 
q(x)=\frac{A^2q\big(\phi^{-1}(x)\big)}{\phi'\big(\phi^{-1}(x)\big)},\ A\in\bbR\backslash\{0\},\, C>0.
\end{align}
$(i)$ The equation satisfied by $r$ in \eqref{require} is Schr\"oder's equation \cite{Schr70}, that is, the equation is the eigenvalue equation for the composition operator sending $f$ to $f\big(\phi^{-1}(\dott)\big)$ with eigenvalue $C^{-1}$ in \eqref{require}.

Moreover, it is interesting to note that $pq$ satisfies the same equation as $r$ with $C=A^4$.\\[1mm]
$(ii)$ If $A=1$ in \eqref{require}, then $p$ and $1/q$ satisfy the same functional equations. Thus the choice $p=1/q$ for $q>0$ is valid in this case. The resulting equation satisfied by $p$ is the so-called Julia's equation \cite{AB15}.

Moreover, letting $\rho=1/p$, the equation for $\rho$ in \eqref{require} can be integrated to yield
$A^2 P(x)= P \big(\phi^{-1}(x)\big),$
where $P'(x)=\pm 1/p(x)$, which is the same Schr\"oder's equation as before with eigenvalue $A^2$ now. Therefore, when $A^{2}=C^{-1}$, the choice $r(x)=\int^x dt/p(t)$ is valid provided the constant of integration is chosen appropriately so that Schr\"oder's equation is satisfied.

Note that if $P(x)$ is finite at an endpoint of the interval, since $\phi^{-1}(d)=d$ for $d\in\{a,b\}$, one must have $P(d)=0$ to be able to choose $A\neq \pm1$ from the functional equation. 
For example, considering $p_\mu(x)=\mu x^2, \mu>0,$ from Example \ref{example}, one must choose $P(x)=\mu^{-1}\big(x^{-1}-1\big)$ so that $P(1)=0$ (with no restriction at $x=0$ since $P$ is infinite there) as $A=A_c=(1+c)^{1/2},\ c>0,$ in this example.

Relating the equations satisfied by the coefficient functions to Schr\"oder's equation is powerful in multiple ways. For instance, for fixed $\phi^{-1}(x)$, if one solves the equation for $P(x)$ and $A$, then $\wti P_n(x)=P^n(x)$ and $\wti A_n=A^n$ define a new pair that solve the equation for the same fixed $\phi^{-1}(x)$. This yields a sequence of new choices for $p(x)$, namely $[p_n(x)]^{-1}=\wti{P}^{\,\prime}_n(x)=nP^{n-1}(x)P'(x)=nP^{n-1}(x)/p(x)$ with $A_n=A^n$, provided $1/p_n\in L^1_{loc}$. 
Returning to Example \ref{example}, this corresponds to $A_{n,c}=(1+c)^{n/2}$ and $p_{n,\mu}(x)=\mu x^2/\big[n \mu^{1-n}\big(x^{-1}-1\big)^{n-1}\big]=n^{-1}\mu^n x^{n+1}\big(1-x\big)^{1-n}$.

Similarly, if $P(x)$ is an invertible solution of Schr\"oder's equation with eigenvalue $s$, one readily verifies that the function $P(x)G(\ln(P(x)))$ is also a solution to the equation for any periodic function $G(x)$ with period $\ln(s)$. In \eqref{require}, this corresponds to the choice $1/\wti{p}(x)=d/dx[P(x)G(\ln(P(x)))]=P'(x)[G(\ln(P(x)))+G'(\ln(P(x)))]$ whenever $G(x)$ is differentiable and $\log(A^2)$-periodic.

One can of course now combine the previous two remarks to yield even more related examples.\\[1mm]
$(iii)$ Similarly, the equation satisfied by $q$ in \eqref{require} can be integrated to yield
$Q(x)=A^2 Q\big(\phi^{-1}(x)\big),$
where $Q'(x)=q(x)$, which is the same Schr\"oder's equation as before with eigenvalue $A^{-2}$ instead. When $A^{2}=C$ this is the same equation satisfied by $r(x)$ yielding that the choice $r(x)=\int^x q(t)dt$ under the same conditions as before.
The other observations in the previous point now hold for $Q(x)$ under the simple change $A\mapsto A^{-1}$.
\end{remark}

In order to study which self-adjoint extensions of a given $K$-invariant symmetric Sturm--Liouville operator remain invariant with respect to $K$, we restrict to the regular setting and recall the following result parameterizing self-adjoint extensions (cf., e.g., \cite[Ch. 4]{GNZ24}, \cite[Sect.~13.2]{We03}, \cite[Ch.~4]{Ze05}):

%%%%%%%%%%%%%
\begin{theorem}\lb{t2.10}
Assume that $\tau$ is regular on $[a,b]$ $($that is, Hypothesis \ref{h2.1} with $L^1_{loc}$ replaced by $L^1$ and finite interval $(a,b)$$)$. Then the following items $(i)$--$(iii)$ hold: \\[1mm] 
$(i)$ All self-adjoint extensions $T_{\al,\be}$ of $T_{min}$ with separated boundary conditions are of the form
\begin{align}
& T_{\al,\be} f = \tau f, \quad f \in \mathcal{D}(T_{\al,\be})=\big\{g\in\mathcal{D}(T_{max}) \, \big| \, g(a)\cos(\al)+g^{[1]}(a)\sin(\al)=0;  \lb{2.25a} \\ 
& \hspace*{5.6cm} g(b)\cos(\be)-g^{[1]}(b)\sin(\be) = 0 \big\},\quad \al,\be\in[0,\pi).    \no  
\end{align}
$(ii)$ All self-adjoint extensions $T_{\eta,R}$ of $T_{min}$ with coupled boundary conditions are of the type
\begin{align}
T_{\eta,R} f = \tau f,\quad f \in \mathcal{D}(T_{\eta,R})=\bigg\{g\in\mathcal{D}(T_{max}) \, \bigg| \begin{pmatrix}g(b)\\g^{[1]}(b)\end{pmatrix} 
= e^{i\eta}R \begin{pmatrix}
g(a)\\g^{[1]}(a)\end{pmatrix} \bigg\}, \ \eta\in[0,\pi),\ R \in SL(2,\bbR).    \lb{2.26a} 
\end{align}
$(iii)$ Every self-adjoint extension of $T_{min}$ is either of type $(i)$ or of type 
$(ii)$.
\end{theorem}
%%%%%%%%%%%%%%
If either endpoint is in the limit point case (which allows for that endpoint to be infinite), the domain of every self-adjoint extension corresponds to \eqref{2.25a} with the separated boundary conditions at that endpoint removed. If both endpoints are in the limit point case, then no boundary conditions are needed as the maximal operator is self-adjoint.
This leads to the following theorem fully describing $K$-invariant self-adjoint extensions.

\begin{theorem}\label{t3.7}
Assume Hypothesis \ref{hypeqs2} and that $\tau$ is regular at each endpoint needing boundary conditions. Then the following  items $(i)$--$(iv)$ hold $($where when one endpoint is limit point, we only consider the case of separated boundary conditions at the other endpoint$)$: \\[1mm]
$(i)$ The only self-adjoint extension that is always $K$-invariant is the extension satisfying the Dirichlet boundary conditions $ g(a)=0= g(b)$. \\[1mm]
$(ii)$ The separated boundary conditions other than Dirichlet which are invariant under $K$, and hence define $K$-invariant self-adjoint extensions, are given as follows:
\begin{align}
\begin{split}\label{3.33}
&\a=\begin{cases}
%0, & \text{if } A^{[1]}(a)\neq0,\ \phi'(a)=1,\\
\pi/2, & \text{if } A^{[1]}(a)=0,\ \phi'(a)\neq1,\\
\cot^{-1}\bigg(\frac{-A^{[1]}(a)}{\big(1-\phi'(a)\big)A(a)}\bigg), & \text{if } A^{[1]}(a)\neq 0,\ \phi'(a)\neq1,
\end{cases}\\
&\a\in(0,\pi),\quad \text{if } A^{[1]}(a)=0,\ \phi'(a)=1,
\end{split}\\
\begin{split}
&\b=\begin{cases}
%0, & \text{if } A^{[1]}(b)\neq0,\ \phi'(b)=1,\\
\pi/2, & \text{if } A^{[1]}(b)=0,\ \phi'(b)\neq1,\\
\cot^{-1}\bigg(\frac{A^{[1]}(b)}{\big(1-\phi'(b)\big)A(b)}\bigg), & \text{if } A^{[1]}(b)\neq 0,\ \phi'(b)\neq1,
\end{cases}\\
&\b\in(0,\pi),\quad \text{if } A^{[1]}(b)=0,\ \phi'(b)=1.
\end{split}
\end{align}
In particular, if $\phi'(a)=\phi'(b)=1$, $A^{[1]}(a)=A^{[1]}(b)=0$, then all separated extensions are $K$-invariant. 

Moreover, if $\phi'(a)=1$ and $A^{[1]}(a)\neq0$, then only the Dirichlet boundary condition $\a=0$ is K-invariant at $a$. An analogous statement holds for the endpoint $b$.
\\[1mm]
$(iii)$ A necessary condition for any coupled boundary condition to be $K$-invariant is $A(a)=A(b)$. Such $K$-invariant coupled extensions are characterized by the following non-mutually exclusive cases:
\begin{align}
\eta\in[0,\pi),\ A(a)=A(b),\ A^{[1]}(a)=0, \text{ and }\ \Bigg\{
\begin{split}\label{eqkcoupled4}
&R_{11}A^{[1]}(b)+R_{21}A(a)\big(\phi'(b)-1\big)=0,\\
&R_{22}A(a)\big(\phi'(a)-\phi'(b)\big)-R_{12}A^{[1]}(b)=0,
\end{split}
\end{align}
or
\begin{align}\label{eqkcoupled4a}
\eta\in[0,\pi),\ A(a)=A(b),\ R_{12}=0, \text{ and }\ \Bigg\{
\begin{split}
&\phi'(a)=\phi'(b),\\
&R_{11}A^{[1]}(b)+R_{21}A(a)\big(\phi'(b)-1\big)-R_{22}A^{[1]}(a)=0.
\end{split}
\end{align}
In particular, if $\phi'(a)=\phi'(b)=1$, $A(a)=A(b)$, and $A^{[1]}(a)=A^{[1]}(b)=0$, then all coupled extensions are $K$-invariant $($regardless of $\eta\in[0,\pi)$$)$.\\[1mm]
$(iv)$ For $A$ constant,  $K$-invariant coupled extensions are characterized as follows:
\begin{equation}\label{eqkcoupled4b}
\begin{cases}
R\text{ such that } R_{21}=0, & \text{if } \phi'(a)=\phi'(b)\neq 1,\\
R\text{ such that } R_{22}=0, & \text{if } \phi'(b)=1\neq\phi'(a),\\
\text{all } R, & \text{if } \phi'(a)=\phi'(b)=1,
\end{cases}\quad \eta\in[0,\pi).
\end{equation}
\end{theorem}

\begin{proof}
Note the boundary value of the function and derivative after the action of $K$ become, respectively,
\begin{equation}\label{kbound}
(Kf)(d)=A(d)f(d),\quad 
(Kf)^{[1]}(d)=A^{[1]}(d)f(d)+A(d)\phi'(d)f^{[1]}(d),\quad d\in\{a,b\},
\end{equation}
provided the quasi-derivative of $A(\dott)$ exists at the endpoint considered (which it does by Hypothesis \ref{hypeqs2}).
The result now follows via Lemma \ref{lemma:A3} by considering whether boundary condition equations in the domains \eqref{2.25a} and \eqref{2.26a} still hold after the action of $K$:

We begin by supposing the separated boundary conditions $\sin(\a)g^{[1]}(a)+\cos(\a)g(a)=0$ for some $\a\in[0,\pi)$. We then must study if, for this same $\a$,
\begin{equation}\label{eqsepk}
\sin(\a)(Kg)^{[1]}(a)+\cos(\a)(Kg)(a)=0.
\end{equation}
Notice that $\a=0$ clearly works, and $\a=\pi/2$ reduces to
\begin{equation}
(Kg)^{[1]}(a)=A^{[1]}(a)g(a)+A(a)\phi'(a)g^{[1]}(a)=A^{[1]}(a)g(a),
\end{equation}
requiring $A^{[1]}(a)=0$ for this to be $K$-invariant.

Otherwise, substituting \eqref{kbound} and $\sin(\a)g^{[1]}(a)=-\cos(\a)g(a)$ into \eqref{eqsepk} yields the requirement
\begin{equation}\label{3.32}
\sin(\a)A^{[1]}(a)+\big(1-\phi'(a)\big)\cos(\a)A(a)=0,\quad \a\in(0,\pi)\backslash\{\pi/2\},
\end{equation}
for the boundary condition to be invariant under $K$. 
Notice that if $\phi'(a)=1$ and $A^{[1]}(a)\neq0$, then we must have $\alpha=0$, that is, Dirichlet boundary conditions. Therefore, \eqref{3.32} reduces to two cases. The first is when  $A^{[1]}(a)=0$ and $\phi'(a)=1$ (noting if one of these hold, both must hold for \eqref{3.32} to hold), in which case all boundary conditions are invariant under $K$ since the value of $\a$ is immaterial. Otherwise, if \eqref{3.32} holds with $A^{[1]}(a)\neq0$, $\phi'(a)\neq1$, we can solve for $\alpha$ to arrive at the second line in \eqref{3.33}
An analogous statement holds for the endpoint $b$, noting the change in sign in the characterization of the boundary conditions leads to a change of sign in the argument of cotangent in \eqref{3.33}. These considerations prove all of the statements regarding separated boundary conditions in the theorem.

The statements for coupled boundary conditions follow similarly.
\end{proof}

Theorem \ref{t3.7} leads to the following intriguing implications:

\begin{corollary}\label{c3.8}
Assume Hypothesis \ref{hypeqs2} and that $T_{min}$ is strictly positive. Then the following  hold: \\[1mm]
$(i)$ If $\tau$ is regular, then $A(a)=A(b)$ as well as $A^{[1]}(a)=0$ and/or $\phi'(a)=\phi'(b)$ must hold.
If $A$ is constant, then
$\phi'(a)=\phi'(b)$ and/or $\phi'(b)=1$ must hold.\\[1mm]
$(ii)$ If $\tau$ is regular, $A$ is constant, and $\phi'(a)=\phi'(b)\neq1$ $($resp., $\phi'(b)=1\neq \phi'(a)$$)$, then the Krein--von Neumann extension must coincide with a coupled extension such that $\eta=0$ and $R_{21}=0$ $($resp., $R_{22}=0$$)$.\\[1mm]
$(iii)$ If $\tau$ is limit point at $x=b$, regular at $x=a$, $A^{[1]}(a)=0$, and $\phi'(a)\neq 1$, the Krein--von Neumann extension must be given via Neumann boundary conditions at $x=a$ $($i.e., $g^{[1]}(a)=0$$)$, whereas if $A^{[1]}(a)\neq 0$ and $\phi'(a)\neq1$, the Krein--von Neumann extension must be given by separated boundary conditions at $x=a$ defined via
\begin{equation}
\a=\cot^{-1}\bigg(\frac{-A^{[1]}(a)}{\big(1-\phi'(a)\big)A(a)}\bigg).
\end{equation}

Analogous statements hold with the endpoints interchanged.\\[1mm]
$(iv)$ If $\tau$ is limit point at $x=b$ and regular at $x=a$, then the case $A^{[1]}(a)\neq0$ and $\phi'(a)= 1$ is not possible if $T_{min}$ is strictly positive. However, if $T_{min}$ is only nonnegative, the case $A^{[1]}(a)\neq0$ and $\phi'(a)= 1$ is admissible, in which case the Krein--von Neumann and Friedrichs extensions must coincide.

An analogous statement holds with the endpoints interchanged.\\[1mm]
$(v)$ If $\tau$ is limit point at $x=b$ and regular at $x=a$, then the norm of the eigenvalue of $K\upharpoonright_{\ker(T_{max})}$ in Theorem \ref{t2.17} is equal to one if and only if $\phi'(a)=1$ and $A^{[1]}(a)=0$.

An analogous statement holds with the endpoints interchanged.
\end{corollary}
\begin{proof}
Items $(i)$ and $(ii)$ are simply implications of the previous theorem, namely $(iii)$ and $(iv)$, and Theorem \ref{thm:A.7} noting that in the quasi-regular and bounded from below case (of which regular is a special case) the Krein--von Neumann extension is always given by coupled boundary conditions \cite[Thm. 3.5]{FGKLNS21}.

Item $(iii)$ and the first part of $(iv)$ now follow by noting that the Krein--von Neumann extension must have zero in its spectrum \cite{FGKLNS21}, whereas the Friedrichs extension is assumed to be strictly positive. The second part of $(iv)$ follows by the previous theorem once again.

Finally, item $(v)$ follows from Theorem \ref{t3.7} $(ii)$ and Theorem \ref{t2.17} $(ii)$.
\end{proof}

\subsection{Examples}\label{sub3.1} We now turn to a few illustrative examples recalling the form of $\tau$ in \eqref{2.1}.

\begin{example}\label{example}
As an example of the previous discussion, consider for $c,\mu\in(0,\infty),\ x\in(0,1)$,
\begin{align}
\begin{split}
&p_\mu(x)=\mu x^2,\quad q(x)=0,\quad r(x)=1,\quad A_c=(1+c)^{1/2},\quad \phi_c(x)=\frac{(1+c)x}{1+cx}.
\end{split}
\label{eq:3.36}
\end{align}
One readily confirms that Hypothesis \ref{hypeqs2} holds with these choices.

Linearly independent solutions to $\tau_{\mu} y=0$ for this example are $u(x)=1$ and $\wti u(x)=x^{-1}$, with the latter not being square integrable at $x=0$. This implies $x=0$ is in the limit point case. Moreover, since $\phi_c'(1)=(1+c)^{-1}\neq1$ for any $c>0$, Theorem \ref{t3.7} $(i)$ and $(ii)$ show the only self-adjoint extensions left invariant under $K$ are those extensions with either Dirichlet or Neumann boundary conditions at $x=1$. By Corollary \ref{c3.8} $(iii)$, these are exactly the Friedrichs and Krein--von Neumann extensions in this case, respectively $($strict positivity follows from the Schr\"odinger form in Example \ref{exampleSchrod} with $\gamma=0$$)$.

Finally, the eigenvalue of $K\upharpoonright_{\ker(T_{max})}$ in Theorem \ref{t2.17} is simply $A_c=(1+c)^{1/2}>1$ since the kernel of $T_{max}$ is spanned by $u(x)=1$. This verifies we are in case $(i)$ of the theorem as expected from above.

\end{example}

We now consider an extension of the previous example utilizing Remark \ref{RemSch}.

\begin{example}\label{example1}
The following satisfies Hypothesis \ref{hypeqs2} with $\phi_c(x)=(1+c)x/(1+cx)$, $c>0$: 
\begin{align}\label{3.60}
&p_{n,\mu}(x)=n^{-1}\mu^n x^{n+1}(1-x)^{1-n},\quad q_{n,\g}(x)=n\g^{n} x^{n-1}(1-x)^{-n-1},\quad r(x)=1,\notag \\
& A_{n,c}=A_c^n=(1+c)^{n/2},\quad n\in\bbN,\ \g\in\bbR,\ c,\mu\in(0,\infty),\ x\in(0,1).
\end{align}
We point out that utilizing Remark \ref{RemSch} allows one to add a litany of weight functions $r$ to this example. For instance, by $(i)$ we can let $r_{\nu}(x)=\nu \big[(1-x)x^{-1}\big]^{\delta}$, $\nu\in(0,\infty),\, \delta\in\bbR$, while $(ii)$ and $(iii)$ yield additional choices. We shall restrict to $r\equiv 1$ and $n\in\bbN$ for simplicity.

The general solution of $\tau_{n,\g,\mu} y=0$ with $p_{n,\mu},q_{n,\g}$ and parameters as in \eqref{3.60} is
\begin{align}
\begin{split}
C_1 \big(x^{-1}-1\big)^{(n/2)\big(1-\sqrt{1+4(\gamma/\mu)^n}\big)}+C_2 \big(x^{-1}-1\big)^{(n/2)\big(1+\sqrt{1+4(\gamma/\mu)^n}\big)},\quad C_1,C_2\in\bbR,
\end{split}
\end{align}
noting that for the case $\g=-4^{1/n}\mu$ with $n$ odd the solutions become linearly dependent so one has to introduce a logarithmic solution. We exclude this case for brevity and assume $\g>-4^{-1/n}\mu$ when $n$ is odd for the argument of the square root to be positive. Notice that only the solution with a negative sign on the square root can possibly be $L^2$ near $x=0$, hence $x=0$ is limit point. Furthermore, this solution will only be $L^2$ near $0$ whenever $\g>-4^{-1/n}n^{-1/n}(2-n^{-1})^{1/n} \mu$ if $n$ is odd, with no restrictions whenever $n$ is even.

We now consider the solution
\begin{equation}
y_{n,\g,\mu}(x)=\big(x^{-1}-1\big)^{(n/2)\big(1-\sqrt{1+4(\g/\mu)^n}\big)}.
\end{equation}
For this solution to also be in $L^2$ near $x=1$ we must have $\g<4^{-1/n}n^{-1/n}(2+n^{-1})^{1/n} \mu.$
Hence, the limit point/limit circle classification at $x=1$ is given as follows:
\begin{align}
x=1\ \text{ is } \begin{cases}
\text{limit circle if } 4^{-1/n}n^{-1/n}\big(2+n^{-1}\big)^{1/n} \mu>\g> -4^{-1/n}\mu,\\
\text{limit point if } \g\geq 4^{-1/n}n^{-1/n}\big(2+n^{-1}\big)^{1/n} \mu,
\end{cases}
\end{align}
where the first lower bound is needed for $n$ odd, and can be replaced with zero $($including equality$)$ when $n$ is even since $\pm\g$ give the same equations/operators in this case.

Furthermore, note that when $\g=0$, the endpoint $x=1$ is regular. Hence Theorem \ref{t3.7} $(ii)$ implies that the only $K$-invariant extension other than Friedrichs is defined by Neumann boundary conditions at $x=1$.
For general $\g$, one can apply \cite[Thm. 3.5]{FGKLNS21} to characterize the Krein--von Neumann extension whenever the $T_{min}$ is strictly positive.

Finally, the eigenvalue of $K\upharpoonright_{\ker(T_{max})}$ in Theorem \ref{t2.17} for this example is $(1+c)^{(n/2) \sqrt{1+4 (\gamma /\mu )^n}}>1$ since the kernel of $T_{max}$ is spanned by $y_{n,\g,\mu}$, thus verifying we are once again in case $(i)$ of the theorem.
\end{example}

We end with an example of an interesting Schr\"odinger operator related to the previous example.

\begin{example}\label{exampleSchrod}
The following satisfies Hypothesis \ref{hypeqs2}:
\begin{align}
\begin{split}
&p(x)=1,\quad q_{\g,\mu}(x)=\dfrac{\g}{\big(1-e^{-\mu^{1/2}x}\big)^{2}}+\dfrac{\mu}{4},\quad r(x)=1,\quad A_{c,\mu}(x)=\big[1+ce^{-\mu^{1/2}x}\big]^{1/2},\\
&\phi_{c,\mu}(x)=-\mu^{-1/2}\ln\bigg[\dfrac{(1+c)e^{-\mu^{1/2}x}}{1+ce^{-\mu^{1/2}x}}\bigg],\quad
\g\in(-\mu/4,\infty),\ c,\mu\in(0,\infty),\ x\in(0,\infty).
\end{split}
\end{align}
In fact, this example can be found via a Liouville--Green transformation performed on Example \ref{example1} with $n=1$ $($see the discussion after this example$)$.
Notice that the potential for this example behaves like a Bessel singularity near $x=0$, but is like the nonzero constant $\g+(\mu/4)$ as $x\to\infty$, unlike the classic Bessel potential which tends to zero at infinity.

Moreover, when $\g=0$ we can apply Theorem \ref{t3.7} $(ii)$ and Corollary \ref{c3.8} $(iii)$ once again to identify the Krein--von Neumann extension since $T_{min}$ is bounded below by $\mu/4$. Note that $A'_{c,\mu}(0)=-c\mu^{1/2}2^{-1}(1+c)^{-1/2}\neq0$ for any $c,\mu$, whereas $\phi'_{c,\mu}(0)=(1+c)^{-1}\neq1$. Therefore, the Krein--von Neumann extension for this example with $\g=0$ is defined via the boundary condition $\a\in(0,\pi/2)$ given by
\begin{equation}
\a=\cot^{-1}\bigg(\frac{-A_{c,\mu}'(0)}{\big(1-\phi_{c,\mu}'(0)\big)A_{c,\mu} (0)}\bigg)=\cot^{-1}\big(2^{-1} \mu^{1/2}\big),\quad \mu\in(0,\infty).
\end{equation}
Once again, one can utilize the notion of generalized boundary values and apply \cite[Thm. 3.5]{FGKLNS21} to characterize the Krein--von Neumann extension whenever the minimal operator is strictly positive.

The eigenvalue of $K\upharpoonright_{\ker(T_{max})}$ in Theorem \ref{t2.17} for this example is $(1+c)^{ \sqrt{1+4 (\gamma /\mu )}/2}\neq1$ since the kernel of $T_{max}$ is spanned by $e^{- \sqrt{\mu } x/2} (e^{\sqrt{\mu } x}-1)^{(1/2)(1-\sqrt{1+4(\g/\mu)})}.$
\end{example}

The last example motivates a closer look at what happens to $K$ and $K^*$ under a Liouville--Green transformation. Under the additional assumptions $(pr),(pr)'/r \in AC_{loc}((a,b))$ and $(pr)\big|_{(a,b)}>0$, the general transformation is of the form (see, e.g., \cite[Thm. 3.5.1]{GNZ24}, \cite[Sect. 4]{FPS25}, and references therein)
\begin{align}
&\xi(x)=\int_k^x  [r(t)/p(t)]^{1/2}\, dt,\quad \cA:=-\int_a^k  [r(t)/p(t)]^{1/2}\, dt,\quad \cB:=\int_k^b  [r(t)/p(t)]^{1/2}\, dt,\quad k\in (a,b),  \notag \\
&u(z,\xi)=[p(x(\xi))r(x(\xi))]^{1/4}y(z,x(\xi)), 
\end{align}
which recasts the equation $\tau y(z,x)=z y(z,x)$ with $x \in (a,b)$ into the form 
\begin{align}
- \frac{d^2}{d\xi^2} u (z,\xi)+V(\xi)u(z,\xi)= zu(z,\xi), \quad \xi \in (\cA,\cB)\subset\R .
\end{align}
The transformed potential $V(\xi)$ can be found to be 
\begin{align}
\begin{split}
V(\xi)&=-\dfrac{1}{16}\dfrac{1}{p(x)r(x)}\left[\dfrac{(p(x)r(x))^{\,\prime}}{r(x)}\right]^2+\dfrac{1}{4}\dfrac{1}{r(x)} \left[\dfrac{(p(x)r(x))^{\,\prime}}{r(x)}\right]^{\,\prime} + \dfrac{q(x)}{r(x)}.
\end{split}
\end{align}
Because of the additional conditions $(pr),(pr)'/r \in AC_{loc}((a,b))$ and $(pr)\big|_{(a,b)}>0$, the potential satisfies $V(\xi)\in L^1_{loc}((\cA,\cB);d\xi)$. 
For example, the $n=1$ case of Example \ref{example1} can be transformed to Example \eqref{exampleSchrod} by choosing
\begin{align}
&\xi(x)=\mu^{-1/2}\int_x^1 t^{-1}\, dt=-\mu^{-1/2}\ln(x), \quad x(\xi)=e^{-\mu^{1/2}\xi},\quad d\xi=-\mu^{-1/2} x^{-1} dx, \notag\\
&u(z,\xi)=\mu^{1/4}(x(\xi))^{1/2}y(z,x(\xi)), 
\end{align}

Next we study the analogs of $K,K^*$ for the transformed operator, which we denote by $\wti K, \wti K^*$.
We will denote the inverse of $\xi(x)$ by $x(\xi)$, the operators associated with this transformed equation by $\wti{T}$, $\boldsymbol{\cdot} =d/d\xi$, and the unitary Liouville--Green transform of a solution from the variable $x$ to $\xi$ defined above by $G$ and its inverse by $G^{-1}=G^*$, that is,
\begin{equation}
(Gf)(\xi)=[p(x(\xi))r(x(\xi))]^{1/4}f(x(\xi)),\quad \big(G^{-1} g\big)(x)=[p(x)r(x)]^{-1/4}g(\xi(x)).
\end{equation}
Assuming \eqref{requirevar}, one then has
\begin{align}
&(\wti{K}f)(\xi)=\big(G K G^{-1} f\big)(\xi)
=C^{-1/4}[A(x(\xi))]^{1/2}\big[\phi'(x(\xi))\big]^{-1/4}f(\xi(\phi(x(\xi)))), \label{Ktilde}\\
&\big(\wti{K}^* f\big)(\xi)=\big(G K^* G^{-1} f\big)(\xi)
= C^{-3/4}
\big[A\big(\phi^{-1}(x(\xi))\big)\big]^{3/2}\big[\phi'\big(\phi^{-1}(x(\xi)\big)\big]^{-3/4} f\big(\xi\big(\phi^{-1}(x(\xi))\big)\big),\\
&\big(\wti{K}^*\wti{T}\wti{K}f\big)(\xi)=\big(G K^* T K G^{-1}f\big)(\xi)=\big(G T G^{-1}f\big)(\xi)=\big(\wti{T} f\big)(\xi).
\end{align}
Notice this naturally transforms the requirements \eqref{require} into \eqref{schrodrequire} since, when $A$ is constant, the new $\wti A$ multiplying $f$ in \eqref{Ktilde} will now depend on the variable in general (unless $\phi$ is linear). Applying these formulas to the $n=1$ case of Example \ref{example1} leads to the form of $K$ in Example \eqref{exampleSchrod}. 

For our last example, we finish by studying a block operator ${\bf T}_{min}$ of the form 
\begin{equation}
{\bf T}_{min}=\begin{pmatrix}T_{min} & 0 \\ 0 & T_{min}\end{pmatrix},
\end{equation}
where $T_{min}$ are the minimal realizations in $L^2(0,1)$ of the Sturm--Liouville differential expression described in Example \ref{example}, with $p(x)=\mu x^2$, $q(x)=0$, and $r(x)=1$, where $\mu>0$ is some fixed positive constant. Choosing ${\bf K}$ as
\begin{equation}
{\bf K}=\begin{pmatrix}0 & K_c \\ K_d & 0\end{pmatrix},
\end{equation}
with $K_c, K_d$ being the similarity transformations corresponding to the choice $A_{c}=(1+c)^{1/2}$, $\phi_c(x)=\frac{(1+c)x}{1+cx}$, and $A_d=(1+d)^{1/2}$, $\phi_d(x)=\frac{(1+d)x}{1+dx}$, respectively, the operator ${\bf T}_{min}$ is ${\bf K}$-invariant.

Note that ${\bf T}_{min}$ is strictly positive and that its defect indices are $(2,2)$, which in addition to its Friedrichs and Krein--von Neumann extensions, will lead to two more ${\bf K}$-invariant nonnegative self-adjoint extensions of ${\bf T}_{min}$, corresponding to the two eigenspaces of ${\bf K}\upharpoonright_{\ker({\bf T}_{max})}$ and choosing $B$ to be the zero operator on these respective spaces (cf.\ Theorem \ref{thm:2.13}).
\begin{example}\label{directsumexample}
Let $\cH=L^2(0,1)\oplus L^2(0,1)=\{(f_1,f_2)\: |\: f_1,f_2\in L^2(0,1)\}$ equipped with the inner product $\langle (f_1,f_2),(g_1,g_2)\rangle_\cH:=\langle f_1,g_1\rangle_{L^2}+\langle f_2,g_2\rangle_{L^2}$. Letting $T_{min}$ be the minimal realization of the Sturm--Liouville differential expression in $L^2(0,1)$ described in Example \ref{example} for some fixed $\mu>0$, we then introduce the strictly positive symmetric operator ${\bf T}_{min}$ given by
\begin{align}
{\bf T}_{min}:\:\cD({\bf T}_{min})=\left\{(f,g)\in \cH\: |\: f,g\in\cD(T_{min})\right\},\:\:(f,g)\mapsto (\tau f, \tau g).
\end{align}
$($Strict positivity follows from the Schr\"odinger form in Example \ref{exampleSchrod} with $\gamma=0$.$)$
The maximal realization ${\bf T}_{max}={\bf T}_{min}^*$ is given by
\begin{align}
{\bf T}_{max}:\:\cD({\bf T}_{max})=\left\{(f,g)\in \cH\: |\: f,g\in\cD(T_{max})\right\},\:\:(f,g)\mapsto (\tau f, \tau g).
\end{align}
Letting $u\in\cD(T_{max})$ be the constant function, $u(x)=1$, which spans $\ker(T_{max})$, the two-dimensional kernel of ${\bf T}_{max}$ is given by
\begin{equation}
\ker({\bf T}_{max})=\left\{(\lambda_1 u, \lambda_2 u)\: | \: \lambda_1,\lambda_2\in\C\right\}.
\end{equation}
Now, for $c,d\in(0,\infty)$, we introduce the bounded and boundedly invertible operators $K_c, K_d$ given by
\begin{equation}
(K_{c}f)(x)=A_{c}f(\phi_{c}(x))\quad\mbox{and}\quad (K_{d}f)(x)=A_{d}f(\phi_{d}(x))
\end{equation}
where $A_{c}=(1+c)^{1/2}, A_d=(1+d)^{1/2}$ and $\phi_c(x)=\frac{(1+c)x}{1+cx}$, $\phi_d(x)=\frac{(1+d)x}{1+dx}$ $($cf.\ Equation \eqref{eq:3.36}$)$. Note that $T_{min}$ is $K_c$ and $K_d$-invariant. Define the operator ${\bf K}$ as
\begin{equation}
{\bf K}: \cH\rightarrow\cH,\quad (f,g)\mapsto (K_cg,K_df),
\end{equation}
with adjoint ${\bf K^*}$ given by 
\begin{equation}
{\bf K^*}: \cH\rightarrow\cH,\quad (f,g)\mapsto (K_d^*g,K_c^*f).
\end{equation}
First note that ${\bf K}\cD({\bf T}_{min})=\cD({\bf T}_{min})$. For every $(f,g)\in\cD({\bf T}_{min})$ we have
\begin{equation}
{\bf K^*}{\bf T}_{min}{\bf K}(f,g)=(K_d^*T_{min}K_df, K_c^*T_{min}K_cg)=(T_{min}f, T_{min}g)={\bf T}_{min}(f,g),
\end{equation}
where we used that $K_{c,d}^*T_{min}K_{c,d}=T_{min}$. Hence, ${\bf T}_{min}$ is ${\bf K}$-invariant. By Lemma \ref{lemma:A.8}, $\ker({\bf T}_{max})={\bf K}\ker({\bf T}_{max})$ and the eigenvectors of ${\bf K}\upharpoonright_{\ker({\bf T}_{max})}$ are given by $v_{\pm}:=(\sqrt{A_c}u, \pm\sqrt{A_d}u)$
with corresponding eigenvalues $\lambda_{\pm}=\pm \sqrt{A_cA_d}$, that is, ${\bf K}v_{\pm}=\lambda_{\pm}v_{\pm}$. Note that $|\lambda_+|=|\lambda_-|>1$.

To describe all nonnegative self-adjoint extensions ${\bf T}_B$ of ${\bf T}_{min}$ that are ${\bf K}$-invariant, there are now three possibilities for choosing the dimension of the domain $\cD(B)$ of the auxiliary operator $B$:\\[1mm]
$(i)$ $\dim(\cD(B))=0$: This corresponds to the Friedrichs extension ${\bf T}_F$ of ${\bf T}_{min}$, which we know by Theorem \ref{thm:A.7} to be ${\bf K}$-invariant. Letting $T_F$ be the Friedrichs extension of $T_{min}$, we have 
\begin{align}
\cD({\bf T}_F)&=\{(f,g)\in\cH\: |\: f,g\in\cD(T_F)\}=\{(f,g)\in\cH\: |\: f,g\in\cD(T_{max}), f(1)=g(1)=0\}.
\end{align}
$(ii)$ $\dim(\cD(B))=2$: In this case, $\cD(B)=\ker({\bf T}_{max})$. Since $|\lambda_{\pm}|>1$, in this case the space $\cC$ defined in \eqref{eq:2.42a} is actually equal to $\cD(B)$. Thus, by Theorem \ref{thm:2.17}, we have $\cD(B)\subseteq\ker(B)\subseteq\cD(B)$, and therefore $B\equiv 0$ on $\ker({\bf T}_{max})$ is the only ${\bf K}$-invariant self-adjoint extension with the property that $\cD(B)=\ker({\bf T}_{max})$. This corresponds to the Krein--von Neumann extension ${\bf T}_K$ of ${\bf T}_{min}$. Letting $T_K$ be the Krein--von Neumann extension of $T_{min}$, we have 
\begin{align}
\cD({\bf T}_K)&=\{(f,g)\in\cH\: |\: f,g\in\cD(T_K)\}=\{(f,g)\in\cH\: |\: f,g\in\cD(T_{max}), f'(1)=g'(1)=0\}.
\end{align}
$(iii)$ $\dim(\cD(B))=1$: In this case we have to choose a one-dimensional subspace of $\ker({\bf T}_{max})$. By Theorem \ref{thm:A.10} it has to satisfy ${\bf K}\cD(B)=\cD(B)$ for ${\bf T}_B$ to be ${\bf K}$-invariant. The only two possible choices for this are the eigenspaces of ${\bf K}\upharpoonright_{\cD(B)}$ given by $\cD(B)=\mbox{span}\{v_+\}$ and $\cD(B)=\mbox{span}\{v_-\}$. If $\cD(B)=\mbox{span}\{v_\pm\}$, the operator ${\bf K}\upharpoonright_{\cD(B)}$ acts just as the multiplication by the scalar $\lambda_\pm$. Since $|\lambda_\pm|>1$, the space $\cC$ defined in \eqref{eq:2.42a} again equals $\mbox{span}\{v_\pm\}$ and thus, by Theorem \ref{thm:2.17}, the only choice for $B$ to describe a ${\bf K}$-invariant nonnegative self-adjoint extension of ${\bf T}_{min}$ is $B\equiv 0$, corresponding to the ${\bf K}$-invariant extensions described in Theorem \ref{thm:2.13}. 
With the choice $\cD(B)=\mbox{span}\{v_\pm\}$ and defining 
\begin{equation}
v_{\pm}^\perp;=(\sqrt{A_d}u,\mp\sqrt{A_c}u),
\end{equation}
we have $\cD(B)^\perp\cap\ker({\bf T}_{max})=\mbox{span}\{v_\pm^\perp\}$. Hence, the two additional ${\bf K}$-invariant nonnegative self-adjoint extensions have the following domains:
\begin{align}
\cD({\bf T}_\pm)=\cD({\bf T}_{min})\dot{+}\mbox{span}\{v_{\pm}\}\dot{+}{\bf T}_F^{-1}\mbox{span}\{v_\pm^\perp\}.
\end{align}
Direct calculation verifies that ${\bf T}_F^{-1}v_{\pm}^\perp$ is given by
\begin{equation}
{\bf T}_F^{-1}v_{\pm}^\perp=(\sqrt{A_d}\log(\dott),\mp\sqrt{A_c}\log(\dott)).
\end{equation}
In terms of boundary conditions, this leads to the following alternative description of $\cD({\bf T}_\pm)$:
\begin{equation}
\cD({\bf T}_\pm)=\{(f,g)\in\cD({\bf T}_{max}) \: | \: \sqrt{A_d}f(1)=\pm \sqrt{A_c}g(1),\:\: \sqrt{A_c}f'(1)=\mp\sqrt{A_d}g'(1)\}.
\end{equation}
Note that these four extensions, ${\bf T}_F,{\bf T}_K,{\bf T}_+$, and ${\bf T}_-$, describe \emph{all} ${\bf K}$-invariant maximally dissipative and nonnegative self-adjoint extensions of ${\bf T}_{min}$.
\end{example}

%\medskip

%%%%%%%%%%%%%%%%%%%%%%%%%%%%%%%%%%%%%%%
{\bf Acknowledgments.} CF is grateful for hospitality at The Ohio State University and financial support through its Mathematics Research institute. He was also supported in part by the National Science Foundation under grant DMS--2510063. BR thanks Baylor University for their hospitality as well. JS was supported in part by an AMS--Simons Travel Grant.

%%%%%%%%%%%%%%%%%%%%%%%%%%%%%%%%%%%%%%%

%%%%%%%%%%%%%%%%%%%%%%%%%%%%%%%%%%%%%%%

\end{document}